\documentclass{birkmult}

\usepackage{amsmath,amsthm,amssymb,mathrsfs,amscd,amsthm,amscd,
amsfonts,graphicx,mathtools, bbm}
\usepackage[cmtip,all]{xy} 	
\usepackage{color}
\usepackage{hyperref}
\hypersetup{
	backref,
	pagebackref,
	colorlinks=true,
	linkcolor=blue,
	filecolor=magenta,      
	urlcolor=cyan,
}

\newcommand{\xrightarrowdbl}[2][]{%
  \xrightarrow[#1]{#2}\mathrel{\mkern-14mu}\rightarrow
} 							
\newcommand{\kf}{{\mathcal F}}
\newcommand{\ko}{{\mathcal O}}

\def\k{\mathbbm{k}}
\def\R{\mathbbm{R}}
\def\C{\mathbbm{C}}
\def\Z{\mathbbm{Z}}
\def\Q{\mathbbm{Q}}
\def\F{\mathbbm{F}}

\def\fm{\frak{m}} \def\fn{\frak{n}}
\def\fp{\frak{p}}\def\fq{\frak{q}}
\def\fa{\frak{a}} \def\fb{\frak{b}} \def\fc{\frak{c}}
\def\x{\langle x \rangle}
\def\px{\langle \fp,x \rangle}

\DeclareMathOperator{\Spec}{Spec}
\DeclareMathOperator{\Max}{Max}
\DeclareMathOperator{\Ann}{Ann}
\DeclareMathOperator{\Sing}{Sing}
\DeclareMathOperator{\Sm}{Sm}
\DeclareMathOperator{\Reg}{Reg}
\DeclareMathOperator{\Supp}{Supp}

\DeclareMathOperator{\Quot}{Quot}

\DeclareMathOperator{\chara}{char}
\DeclareMathOperator{\coker}{coker}

\newtheorem{theorem}{Theorem}
\newtheorem{lemma}[theorem]{Lemma}
\newtheorem{corollary}[theorem]{Corollary}

\newtheorem{proposition}[theorem]{Proposition}
\newtheorem{definition}[theorem]{Definition}
\newtheorem{notation}[theorem]{Notation}
\newtheorem{example}[theorem]{Example}
\newtheorem{remark}[theorem]{Remark}

\begin{document}
\title[Semicontinuity of Singularity Invariants]
 {Semicontinuity of Singularity Invariants \\
 in Families of Formal Power Series}
\author[Greuel and Pfister]{Gert-Martin Greuel and Gerhard Pfister}
\dedicatory{Dedicated to Andr\'as N\'emethi on the occasion of his sixtieth birthday}

\address{%
University of Kaiserslautern\\
Department of Mathematics\\
Erwin-Schroedinger Str.\\
67663 Kaiserslautern\\
Germany}

\email{greuel@mathematik.uni-kl.de}
\email{pfister@mathematik.uni-kl.de}


\bigskip
\begin{abstract}
The problem we are considering  came up in connection with the classification of singularities in positive characteristic. 
Then it is important that certain invariants like the determinacy can be bounded simultaneously in families of formal power series parametrized by some algebraic variety. In contrast to the case of analytic 
or algebraic families, where such a bound is well known, the problem is rather subtle, since the modules defining the invariants are quasi-finite but not finite over the base space. In fact, in general the fibre dimension is not semicontinuous and the quasi-finite locus is not open. However, if we pass to the completed fibres in a family of rings or modules we can prove that  their fibre dimension  is semicontinuous under some mild conditions. We prove this in a rather general framework by introducing and using the completed and the Henselian tensor product, the proof being more involved than one might think. Finally we apply this to the Milnor number and the Tjurina number in families of hypersurfaces and complete intersections and to the determinacy in a family of ideals.

\end{abstract}
\maketitle

\renewcommand{\contentsname}{Table of Contents}
\tableofcontents

\section*{Introduction}

In connection with the classification of singularities defined by formal power series  over a field a fundamental invariant is the modality of the singularity (with respect to some equivalence relation like right or contact equivalence). To determine the modality one has to investigate adjacent singularities that appear in nearby fibres. This cannot be done by considering families over complete local rings but one has to consider families of power series parametrized by some algebraic variety in the neighbourhood of a given point.
To determine potential adjacencies, an important tool is the semicontinuity of certain 
singularity invariants like, for example, the Milnor or the Tjurina number.  Another basic question is if the determinacy of an ideal can be bounded by a semicontinuous invariant. In the complex analytic situation the answer to these questions is well known and positve, but for formal power series the problem is much more subtle than one might think at the first glance. This is mainly due to the fact that ideals or modules that define the invariants are quasi-finite but not finite over the base space.

The modality example shows that the questions treated in this paper are rather natural and appear in important applications. 
Moreover, the semicontinuity in general is a very basic property with numerous applications in many other contexts. Therefore we decided to choose a rather general framework with families of modules presented by matrices of power series and parametrized by the spectrum of some Noetherian ring. It is not difficult to see that the fibre dimension is in general not semicontinuous and that the quasi-finite locus is in general not open (in contrast the case of ring maps of finite type, where the quasi-finite locus is open by Zariski's Main Theorem, cf. Proposition \ref{prop.ftype}), see Examples \ref{ex.Kt} and \ref{ex.Z}. It turns out that the situation is much more satisfactory if we consider not the fibres but the completed fibres and we prove the desired semicontinuity for the completed fibre dimension under some conditions on the family. To guarantee that the completed fibre families behave well under base change we introduce the notion of a (partial) completed tensor product and study its properties in sections \ref{ssec.ctp} and \ref{ssec.qfm}.  

Unfortunately, we cannot prove the semicontinuity of the completed fibre dimension in full generality. We prove it if either the base ring has dimension one (in section \ref{ssec.1dim}),
or if the base ring is complete local containing a field, 
 or if the presentation matrix has polynomials or algebraic power series as entries (in section \ref{ssec.poly}). Together, these cases cover most applications (see Corollary \ref{cor.qf} for 
 a summary). To treat the latter case, we use Henselian rings and the Henselian tensor product, for which we give a short account in section \ref{ssec.hs}. It would be interesting to know, if the result holds for presentation matrices with arbitrary power series as entries or if there are counterexamples. In section \ref{ssec.related} we consider also the case of families of finite type over the base ring and prove a version of Zariski's main theorem for modules.
Moreover, we compare the completed fibre with the usual fibre.

In section \ref{sec.singinv} we apply our results to singularity invariants. We discuss and compare first the notions of regularity and smoothness (over a field) and show that both notions coincide for the completed fibres (Lemma \ref{lem.sing}). 
Under the restrictions mentioned above, we prove the semicontinuity of the Milnor number and Tjurina number for hypersurfaces (section \ref{ssec.mutau}) and  the Tjurina number for complete intersections  (section \ref{ssec.icis}) as well as an upper bound for the determinacy of an ideal  (section \ref{ssec.det}). Since the base ring may be the integers, our results are of some interest for computational purposes. For example, if a power series has integer coefficients then the Milnor number over the rationals is bounded by the Milnor number modulo just one (possibly unlucky) prime number if this is finite (see Corollary \ref{cor.mutau} and, more generally,  Corollary \ref{cor.Z} and Remark \ref{rem.Z}). 

We assume all rings to be associative, commutative and with unit. Throughout the paper  $\k$ denotes an arbitrary field, $A$ a ring, $R= A[[x]], x = (x_1,\cdots, x_n)$, the formal power series ring over $A$
 and $M$ an $R$-module. 
 For our main results we will assume that $A$ is Noetherian and that $M$ is finitely generated as $R$-module.\\
 
 {\bf Acknowledgment:} We would like to thank the anonymous referee for their careful proofreading and in particular for providing answers to two of our original questions (cf. the comment after Proposition \ref{prop.ann}).

\section{Quasi-finite modules and semicontinuity}\label{sec.qms}
\medskip

\subsection {The completed tensor product}\label{ssec.ctp}
Let $A$ be a ring, $R=A[[x]]$ and $M$ an $R$-module.
For any prime ideal  $\fp$ of $A$ let $k(\fp) = A_\fp/\fp A_\fp $ be the residue field of $\fp$. $k(\fp) = \Quot(A/\fp)$ is the quotient field of $A/\fp$ and hence  $k(\fp) = A/\fp$  if $\fp$ is a maximal ideal. 
We consider $M$ via the canonical map $A  \xhookrightarrow {\text{} } R$ as an $A$-module and set
$$M(\fp):= M_\fp \otimes_{A_\fp} k(\fp) =(M \otimes_A A_\fp) \otimes_{A_\fp} k(\fp) = M \otimes_{A} k(\fp),$$
which is called the {\em fibre} of $M$ over $\fp$. $M(\fp)$ is a vector space over $k(\fp)$ 
and its dimension is denoted by
$$ d_\fp (M) := \dim_{k(\fp)} M(\fp).$$
$M$ is called {\em quasi-finite}\footnote{For $M=R/I$, $I$ an ideal, this is the original definition of Grothendieck. Nowadays most authors (e.g. \cite{Sta19}) require in addition that $R$ is of finite type over $A$.} over $\fp$ if $ d_\fp (M) <\infty$.
We are interested in the behavior of $\ d_\fp (M)$ as $\fp$ varies in $\Spec A$, in particular in finding conditions under which $\ d_\fp (M)$ is semicontinuos on $\Spec A$.

We say that  a function $d: \Spec A \to \R$,  $\fp \mapsto d_\fp$,
 is {\em (upper) semicontinuous} at $\fp$ if $\fp$ has an open neighbourhood $U \subset \Spec A$  such that  $d_\fq  \leq  d_\fp $ for all $\fq \in U$. 
 $d $ is semicontinuous on  $\Spec A$ if it is semicontinuous at every
$\fp \in \Spec A$.
\medskip

For finitely presented $A$-modules $M$ the semicontinuity of 
$\fp \mapsto d_\fp (M)$  is true and well known (cf. Lemma \ref{lem.Afin}). However, in many applications $M$ is not finitely generated over $A$ but finite over some $A$-algebra  $R$. Such  a situation appears naturally in algebraic geometry, when one considers families of schemes or of coherent sheaves over $\Spec A$.
But then it is usually assumed that the ring $R$ is either of (essentially) finite type over $A$ (in algebraic geometry) or an analytic $A$-algebra (in complex analytic geometry). When we study  families of singularities defined by formal power series (cf. Section \ref{sec.singinv}), we have to consider  $R=A[[x]]$, which is not of finite type over $A$. As far as we know, this situation has not been systematically studied and it leads to some perhaps unexpected results. For example, $\, d_\fp (M)$ is in general not semicontinuous on  $\Spec A$ (cf. Examples \ref{ex.Kt},  \ref{ex.Z}).   

It turns out that the situation  is much more satisfactory if we pass from the usual fibres to the completed fibres, that is, we consider the {\em completed fibre dimension}
$$ \hat d_\fp (M) := \dim_{k(\fp)} M(\fp)^\wedge,$$
where $M(\fp)^\wedge$ is the $\x$-adic completion of the $R(\fp)$-module $M(\fp)$. To guarantee that the completed fibres behave well when $\fp$ varies in $\Spec A$, we introduce the notion of a completed tensor product below.

\bigskip

 For a finitely presented $A$-module $M$   the semicontinuity of  $\fp \to d_\fp(M)$ is well known:
 
\begin{lemma}\label{lem.Afin}
If $M$ is a finitely presented $A$-module then $\ d_\fp (M)$ is semicontinuous on   $\Spec A$. Moreover, if $M$ is $A$-flat, then $\ d_\fp (M)$ is locally constant on  $\Spec A$.
\end{lemma}

\begin{proof} Fix $\fp \in \Spec A$ and consider a presentation of $M$,
$$ A^p \xrightarrow{\text {$P$}} A^q \to  M \to 0,$$
with matrix $P = (p_{ij}), \ p_{ij}  \in A$. Applying $\otimes_A k(\fp)$ to this sequence we get the exact sequence of vector spaces
 $$ k(\fp)^p \xrightarrow{\text {$\overline P_\fp$}} k(\fp)^q \to  M(\fp) \to 0,$$
with entries of $\overline P_\fp$ being the images of $p_{ij}$ in $k(\fp)$. 
Then  $\ d_\fp (M)$ is finite and
$\ d_\fp (M)  = q - rank(\overline P_\fp)$. Since $rank(\overline P_\fp) \leq rank(\overline P_\fq)$ for all $\fq$ in some neighbourhood $U$ of $\fp$, the claim follows.

If $M$ is flat, then $M_\fp$ is free over the local ring $A_\fp$ for a given $\fp \in \Spec A$. By  \cite{Mat86}, Theorem 4.10 (ii) (and its proof) there exists an $f \notin \fp$ such that $M_f$ is a free $A_f$-module of some rank $r$ and hence
$d_\fq (M) =r$ for  $\fq$ in  the open neighbourhood $D(f)$ of $\fp$.
\end{proof}
\bigskip

We introduce now the completed tensor product. Let us denote by
$$\langle x \rangle :=  \langle x_1,...,x_n \rangle_R$$ 
  the ideal in $R$ generated by $x_1,...,x_n$. More generally, if
$S$ is an $R$-algebra, then   $\langle x \rangle_S$  denotes  the ideal in $S$ generated by (the images of) $x_1,...,x_n$.

For an $R$-module $N$ denote by
$$N^{\wedge} :=  \lim\limits_{\longleftarrow}N/\langle x \rangle^mN $$ 
 the $\langle x \rangle$-adic completion of $N$. If $N$ is also an $S$-module for some $R$-algebra $S$, then 
$\langle x \rangle^m N =( \x_S)^m N$, and hence the 
$\x$-adic completion and the $\x_S$-adic completion of $N$ coincide.

\begin{definition}\label{def.ctp}
Let $A$ be a ring, $R=A[[x]]$,  $B$ an $A$-algebra and $M$ an $R$-module. We define  the {\em completed tensor product} of $R$  and $B$ over $A$ as
the ring
$$R \hat\otimes_A B :=  \lim\limits_{\longleftarrow}\big((R/\langle x \rangle^m) \otimes_A B \big)$$
and the  {\em completed tensor product} of  $M$  and $B$ over $A$ as the module 
$$M \hat\otimes_A B :=  \lim\limits_{\longleftarrow}\big((M/\langle x \rangle^m M)\otimes_A B\big).$$
If $N$ is an $A$-module, we define  the  $R$-module
$$M \hat\otimes_A N :=  \lim\limits_{\longleftarrow}\big((M/\langle x \rangle^m M)\otimes_A N\big)$$
and call it the {\em completed tensor product} of  $M$  and $N$ over $A$.
\end{definition}

One reason why we consider the completed tensor product is that it provides the right base change property in the category of rings of the form $A[[x]]$ by the following Proposition \ref{prop.pctp}.1.

\begin{proposition}\label{prop.pctp}
The completed tensor product has the following properties (assumptions as in Definiton \ref{def.ctp}).
\begin{enumerate}
\item
$A[[x]] \hat\otimes_A B = (R \otimes_A B)^{\wedge} = B[[x]].$ 
\medskip

\item $M \hat\otimes_A N = (M \otimes_A N)^{\wedge}.$
\medskip

\item If $M$ is finitely presented over $R$ and $N$ is a finitely presented  $B$-module, then
$$M \hat\otimes_A N \cong 
(M  \otimes_A N)\otimes_{R\otimes_A B}‚ (R\hat\otimes_A B).$$ 
\item 

The canonical map
 $M \otimes_A N \to M \hat\otimes_A N$  is injective if $A$ is Noetherian, $M$ finite over $R$ and $N$ finite over $A$.
\medskip
\item If $\x ^m \subset \Ann_R (M)$ for some $m$  then
$M \hat\otimes_A N = M \otimes_A N$ for every $A$-module $N$. 

 \end{enumerate}
\end{proposition}

\begin{proof}
1. We have $ \lim\limits_{\longleftarrow}\big(A[[x]]/\langle x \rangle^m \otimes_A B \big)$ 
= $ \lim\limits_{\longleftarrow}\big(A[x]/\langle x \rangle^m \otimes_A B \big)$
= $  \lim\limits_{\longleftarrow}B[x]/\langle x \rangle^m $ = $B[[x]]$,
showing the second equality. The first follows since  $(R/\langle x \rangle^m) \otimes_A B = (R\otimes_A B)/\langle x \rangle^m ({R\otimes_A B})$.

2. Since $(M/\langle x \rangle^m M)\otimes_A N 
= (M\otimes_A N)/\langle x \rangle^m(M\otimes_A N)$ the equality follows and that $(M \otimes_A N)^{\wedge}$ is the $\x_R$ as well as the  
$\x_{R\otimes_A B}$-adic completion of $M\otimes_A N$.

3. If $M$ resp. $N$ are finitely presented over $R$ resp. $B$, then 
$M\otimes_A N$ is finitely presented over $R\otimes_A B$. 
Hence we can apply (the proof of)  \cite[Proposition 10.13]{AM69} and use 1. to show the isomorphism.

4. If $A$ is Noetherian then $R$ is Noetherian. If
$M$ is finitely generated over $R$ and $N$ finitely generated over $A$ then $M\otimes_A N$ is finitely generated over $R$. 
The injectivity follows from 2. and \cite[Theorem 10.17 and Corollary 10.19]{AM69},
since $\x$ is contained in the Jacobson radical of $R$ by Lemma \ref{lem.maxA}.

5. If $\x^m M = 0$ for some $ m$, then $M \hat\otimes_A N = M \otimes_A N$ by definition of the completed tensor product.
 \end{proof}
 
\begin{corollary}\label{cor.rex}
The completed tensor product is right-exact on the category of finitely presented modules. That is, let  
$$
\begin{array}{*{20}l}
M' \to M \to  M'' \to 0 ,  \text{resp.}\\
N' \to N \to N'' \to 0
\end{array}
$$
be exact sequences of finitely presented $R$-modules resp. $B$-modules. Then the sequences of finitely presented $R \hat\otimes_A B$-modules
$$
\begin{array}{*{20}l}
M' \hat\otimes_A N\to M\hat\otimes_A N\to M'' \hat\otimes_A N\to 0,\text{resp.}\\
M \hat\otimes_A N'\to M\hat\otimes_A N\to M\hat\otimes_A N''\to 0\\
\end{array}
$$
are exact. 
\end{corollary}

\begin{proof}
The sequences 
 $ M'\otimes_A N \to M\otimes_A N\to M'' \otimes_A N\to 0$ and 
  $ M\otimes_A N' \to M\otimes_A N\to M\otimes_A N''\to 0$ 
are exact. Now tensor these sequences with $R\hat\otimes_A B$ over 
$R\otimes_A B$  and apply Proposition \ref{prop.pctp}.3. 
 \end{proof}

\begin{corollary}\label{cor:cpl} \noindent
 \begin{enumerate}
\item[(i)] 
 $R\hat\otimes_A A =R$.
 \medskip
 
 \item [(ii)] If $S$ is multiplicatively closed in $A$ then 
 $A[[x]]\hat\otimes_A (S^{-1}A) = (S^{-1}A)[[x]]$.
 \medskip
 
 \item[(iii)] For any $R$-module $M$ we have $M\hat\otimes_A A =M^\wedge$. 
\medskip
 
 \item[(iv)] If $M$ is  finitely presented  over $R$ then $M=M^\wedge$. If moreover  $N$ is finitely presented over $A$, then  $M\hat\otimes_A N =M\otimes_A N$. 
\end{enumerate}
\end{corollary}

\begin{proof}
(i) and (ii) follow immediately  from Proposition \ref{prop.pctp}.1,
(iii) is a special case of Proposition \ref{prop.pctp}.2 and
(iv) follows from (i) and Proposition \ref{prop.pctp}.3 with $B=A$.
  \end{proof}
 
Applying Corollary \ref{cor.rex} and Proposition \ref{prop.pctp}.1 we get
\begin{corollary}\label{cor.pres}
If  $ A[[x]]^p \xrightarrow{\text {$T$}} A[[x]]^q \to  M \to 0$ is an  $A[[x]]$-presentation of $M$ and $B$ an $A$-algebra, then 
$$M\hat\otimes_A B=\coker \big(B[[x]]^p \xrightarrow{\text {$T$}}B[[x]]^q\big).$$
\end{corollary}

 \begin{remark} \label{rem.faith}{\em 
 Let $A$ be Noetherian and $B$ an $A$-algebra. If $\x$ is contained in the Jacobson radical of $R\otimes_A B$, then $R\hat\otimes_A B$ is
 faithfully flat over $R\otimes_A B$, by Proposition \ref{prop.pctp}.1 and \cite[Theorem 8.14]{Mat86}. Note however that although $\x$ is contained in the Jacobson radical of $R$ by Lemma \ref{lem.maxA} below, it need not be in the  Jacobson radical of $R\otimes_A B$ (cf. Example  \ref{ex.2}). 
  }
 \end{remark}
  
 \begin{example} \label{ex.1}{\em 
Let $\langle f_1,\dots,f_k\rangle \subset R=A[[x]]$ be an ideal and $M=A[[x]]/\langle f_1,\dots,f_k\rangle$. 
If $\fp$ is a prime ideal in $A$ then $R\hat \otimes_A A_\fp = A_\fp[[x]]$ and form Corollary \ref{cor.pres} we get  $M\hat\otimes_A A_\fp = A_\fp[[x]]/\langle f_1,\dots,f_k\rangle.$ If $k(\fp)$ is the residue field of $A$ at $\fp$ then 
$M\hat\otimes_A k(\fp) = k(\fp)[[x]]/\langle f_1,\dots,f_k\rangle$, something what one expects as fibre of $M$ over $\fp$. While
$A_\fp[[x]]$ and $k(\fp)[[x]]$ are nice local rings, the subrings $R\otimes_A A_\fp \subset A_\fp[[x]]$ and  $R\otimes_A k(\fp) \subset k(\fp)[[x]]$ are in general not local if $\fp$ is not a maximal ideal (see Example \ref{ex.2}).
}
 \end{example}
\begin{remark}{\em
 Proposition \ref{prop.pctp}.1 with $B=A[[y]], \, y=(y_1,...,y_m),$ implies
 $$ A[[x]] \hat\otimes_A A[[y]] =A[[x,y]].$$
Now let $A$ be Noetherian. If $I$ resp. $J$ are ideals in $A[[x]]$ resp. $A[[y]]$,  we get from Corollary \ref{cor.rex}
 $$ A[[x]]/I \hat\otimes_A A[[y]]/J =A[[x,y]]/\langle I,J \rangle {A[[x,y]]}.$$
We call an $A$-algebra a {\em formal $A$-algebra} if it is isomorphic to an $A$-algebra $A[[x]]/I$.  For two formal $A$-algebras $B=A[[x]]/I$ and $C=A[[y]]/J$  the completed tensor product can be defined as 
$B\hat\otimes_AC = A[[x,y]]/\langle I,J \rangle{A[[x,y]]}.$
It has the usual universal property of the tensor product in the category of formal $A$-algebras, analogous to the analytic tensor product for analytic algebras (cf. \cite[Chapter III.5]{GR71}). Thus, Definition \ref{def.ctp} generalizes the completed tensor product of formal $A$-algebras.}
\end{remark} 
\medskip

\subsection {Fibre and completed fibre}\label{ssec.qfm}

Let again $A$ be a ring and $M$ an $R=A[[x]]$-module. We introduce the completed fibre $\hat M(\fp)$ and the completed fibre dimension  $\hat d_\fp M$ of 
$M$ for $\fp \in \Spec A$ and compare it with the usual fibre $M(\fp)$ and the usual fibre dimension  $d_\fp M$. 

At the end of this section we give examples, showing that semicontinuity of $d_\fp (M)$ does not hold in general on $ \Spec A$, even if $A =\C[t]$ or $A=\Z$ (Examples \ref{ex.Kt} and \ref{ex.Z}).  
However, we show in the next sections \ref{ssec.1dim} and \ref{ssec.poly}  that, under some conditions, semicontinuity holds for the completed fibre dimension $\hat d_\fp (M)$.

\begin{notation}\label{nota} {\em
We have canonical maps 
$$A  \xhookrightarrow{\text {$j$}} R \xrightarrowdbl {\text {$\pi$}} R/\langle x \rangle \xrightarrow[\text {$\cong$}]{\text {$i$}} A,$$
with $i\circ\pi\circ j = id$ and for an ideal $I \subset R$ we set $\overline I := \pi(I).$
On the level of schemes we have the maps 
$\Spec A \xrightarrow[\text {$\cong$}]{\text {$i^*$}} V(\langle x \rangle)
  \xhookrightarrow{\text {$\pi^*$}}  \Spec R \xrightarrowdbl {\text {$j^*$}} \Spec A,$
with $i^*(\fp) = \px,  \  \
j^*(\px) = \px \cap A =\fp$ \ for  $\fp \in \Spec A$.  
We denote by
$$\fn_\fp := \langle \fp, x\rangle =  \langle \fp, x_1,...,x_n\rangle_R $$
the ideal in $R$ generated by $\fp  \in \Spec A$ and $x_1,...,x_n$.
The family
$$f:= j^*: \Spec R \to \Spec A$$
has the trivial section $ \sigma =(i\circ \pi)^*: \Spec A \to \Spec R$, $\fp \mapsto \fn_\fp$,
and the composition $h:=\pi\circ j: A\cong R/\x$ induces an isomorphism  
$$ h^*: V(\x) \xrightarrow {\text {$\cong$}} \Spec A,$$ 
the restriction of $f$ to $V(\x)$.  
We call $R_\fp := R\otimes_A A_\fp$  the {\em stalk} of $R$  over $\fp$. 
$R_\fp$ is not  a local ring, its local ring at $\fn_\fp$ is  
$(R_\fp)_{\fn_\fp} = R_{\fn_\fp}$ with residue field $k(\fn_\fp) = k(\fp)$ 
(by Lemma \ref{lem.cf} below).
} 
\end{notation}

If $M$ is an $R$-module, we call $M_\fp = M\otimes_A A_\fp$  the {\em stalk} of $M$  over $\fp$ and we are interested in the behavior of $M$ along the section 
$\sigma$. However, we are not interested in the $R(\fp)$-modules $M(\fp)$ since
$R(\fp)$ is not a power series ring (and does not behave nicely).  We are interested in the completed stalk $ \hat M_\fp$ and in the
completed fibres $\hat M(\fp)$, which we introduce now. 

\begin{definition} \label{def.cf}
Let $A$ be a ring, $R=A[[x]]$, $M$ an $R$-module and  $\fp \in \Spec A$. 
\begin{enumerate}
\item We set  $\hat R_\fp :=R\hat\otimes_A A_\fp$, a local ring isomorphic to $A_\fp[[x]]$ (Proposition \ref{prop.pctp}.1), and call the 
$\hat R_\fp$-module
$$ \hat M_\fp := M \hat\otimes_{A} A_\fp$$
 the  {\em completed stalk} of $M$  over $\fp$.
\item The ring $ \hat R(\fp) := R\hat\otimes_A k(\fp) $ is called the  {\em completed fibre of $R$  over $\fp$}. It is  a local ring isomorphic to
 $ k(\fp)[[x]]$ (Proposition \ref{prop.pctp}.1).
The $\hat R(\fp)$-module
$$ \hat M(\fp) := M \hat\otimes_{A} k(\fp) = \hat M_\fp \otimes_{A_\fp} k(\fp)$$
is called the  {\em completed fibre of $M$  over $\fp$.} 
\item $ \hat M(\fp)$ is a $k(\fp)$-vector space and we call its dimension
$$\hat d_\fp (M) := \dim_{k(\fp)} \hat M(\fp)$$
the {\em completed fibre dimension of $M$ over $\fp$}.
\item $M$ is called {\em quasi-completed-finite over $\fp$} if $\hat d_\fp (M) < \infty$.
\end{enumerate}
\end{definition}

The map $A \to R$ induces a map of local rings $A_\fp \to R_{\fn_\fp}$ and for an $R$-module $M$ we have the fibre 
$M(\fp)=M\otimes_{A} k(\fp)$ of $M$ w.r.t. $A \to R$
and the fibre 
$$M_{\fn_\fp}(\fp) = M_{\fn_\fp}\otimes_{A_\fp} k(\fp)=M_{\fn_\fp}/\fp M_{\fn_\fp}$$ 
of $M_{\fn_\fp}$ 
w.r.t. $A_\fp \to R_{\fn_\fp}$. The fibres are in general different but the completed fibres coincide by  Lemma \ref{lem.loc} if $M$ is finitely $R$-presented.

\medskip

Let us first compare the fibre $M(\fp)$  with its completed fibre $\hat M(\fp)$. 
\medskip

\begin{lemma}\label{lem.cf}\begin{enumerate}
For  any $R$-module $M$ the following holds.
\item [(i)]    
$\hat M_\fp = (M_\fp)^\wedge$ and $ \hat M(\fp)  =   M(\fp)^\wedge$.

\item [(ii)] $\fn_\fp$  is a prime ideal in $R$ with $\fn_\fp \cap A = \fp $ 
and the residue field of $\fn_\fp$ in $R$ satisfies $k({\fn_\fp}) = k(\fp).$
\item [(iii)] If  $\fn$ is any prime ideal in $R$ containing $\x$, then $\fn = \fn_\fp$
with $\fp = \fn \cap A \in \Spec A$.
\end{enumerate}
\end{lemma}

\begin{proof}
Statement (i) follows from Proposition \ref{prop.pctp}.1. The first statement of
(ii)  follows since $R/\fn_{\fp} = A/\fp $ is an integral domain. Since $R/ \fn_\fp=A/\fp$
we have $k({\fn_\fp}) =  \Quot(R/ \fn_\fp) = \Quot(A/\fp) = k(\fp). $  (iii) is obvious.
\end{proof}

\begin{remark}
{\em We have strict flat inclusions $A_\fp \subsetneqq R_\fp \subsetneqq R_{\fn_\fp} \subsetneqq A_\fp[[x]]$ of rings that are Noetherian if $A$ is Noetherian.
 
The strictness is easy to see. E.g. 
$ g_0 + \sum_{|\alpha | \geq 1} (g_\alpha/h_\alpha) x^\alpha, \ g_0 \notin \fp$,
with arbitrary $h_\alpha \in R \smallsetminus \fn_\fp$, is a unit in  $A_\fp[[x]]$ but it is not contained in $R_{\fn_\fp}$,
where only finitely many different denominators are allowed.
We have $R_\fp = S^{-1} R$, with  $S$ the multiplictive set $A\smallsetminus \fp$ and $R_{\fn_\fp} = (R_\fp)_{\fn_\fp}$.   
Since localization preserves flatness (\cite[Corollary 3.6]{AM69}) and the Noether property (\cite[Proposition 7.3]{AM69}), the inclusions  $A_\fp \subset R_\fp \subset R_{\fn_\fp}$  are flat and the rings are Noetherain if $A$ is Noetherain.
The flatness of $A_\fp[[x]]$  over  $R_{\fn_\fp}$ follows, since the first  is the $\langle x \rangle$-adic completion of the second by Lemma \ref{lem.cf} (i). Since both rings are local, $R_{\fn_\fp} \subset A_\fp[[x]]$
is faithfully flat.

The rings  $R_\fp$ and $R_{\fn_\fp}$ are ``strange'' subrings of $A_\fp[[x]]$. The ring $A_\fp[[x]]$  is of interest in applications (cf. section \ref{sec.singinv}),  while the rings $R_\fp $ and $R_{\fn_\fp}$ are of minor interest. By the following Lemma \ref{lem.loc} we have $(R_\fp )^\wedge = (R_{\fn_\fp})^\wedge =A_\fp[[x]].$}
\end{remark}

\begin{example}\label{ex.2}{\em
As an example let $A=\k[t]$ and $R=A[[x]]$ with $t$ and $x$ one variable, $\fp =\langle 0 \rangle \in  \Spec A$. We have $A_\fp = k(\fp) = \k(t)$ and
$$R_\fp  = \k[t][[x]] \otimes_{\k[t]} \k(t) =  \{g / h \, | \, g \in \k[t][[x]],  h \in \k[t]\smallsetminus 0\},$$
$g = g_0 + \sum_{i \geq 1} g_i x^i, \ g_i \in \k[t]$, a subring strictly contained in  $R\hat\otimes_A A_\fp  = \k(t)[[x]].$
\begin{itemize}
\item $\x$ is contained in the Jacobson radical of  $R\hat\otimes_A A_\fp $ by Lemma \ref{lem.maxA}.
\item $\x$ is not contained in the Jacobson radical of $R_\fp = R\otimes A_\fp$.\\
To see this, note that the element $t-x$  is a unit in  $R\hat\otimes_A A_\fp $, 
 since $1/(t-x) = 1/t \sum_{i \geq 0} (x/t)^i$, but $1/(t-x)$ is not an element in $R_\fp $. 
 The ideal $\langle t-x\rangle$ is a maximal ideal in $R_\fp$, since  
 $R_\fp/\langle t-x\rangle \cong \k((t))$ (see Example \ref{ex.Kt}.2), but
 $x\notin \fm$ since otherwise $t \in \fm$, contradicting the fact that $t$ is a unit $R_\fp $. 
\item The rings $R_\fp$ and $R(\fp)$ are in general not local.\\
 Since $R_\fp/\x =\k(t)$, the ideals $\x$ and  $\langle t-x\rangle$ are two different maximal ideals and  $R_\fp =R(\fp)$  ($\fp =\langle 0 \rangle$) is not local.
\end{itemize}
}
\end{example}

\begin{lemma} \label{lem.loc}
Let $M$ be a finitely presented $R$-module and $\fp \in \Spec A$.
\begin{enumerate} 
\item We have isomorphisms
$$  \hat M_\fp    \cong M_{\fn_\fp} \hat\otimes_{A_\fp} A_\fp= M_{\fn_\fp} \hat\otimes_{A} A= (M_{\fn_\fp})^\wedge.$$
\item  $\hat M(\fp) \cong (M_{\fn_\fp}/\fp M_{\fn_\fp})^\wedge$.
\item If $M=\coker \big(A[[x]]^p \xrightarrow{\text {$T$}}A[[x]]^q\big)$ then
$$
\begin{array}{*{20}c}
\hat M_\fp  = M\hat\otimes_A A_\fp =\coker \big(A_\fp[[x]]^p \xrightarrow{\text {$T$}}A_\fp[[x]]^q\big),\\
\hat M(\fp) =\coker \big(k(\fp)[[x]]^p \xrightarrow{\text {$T$}}k(\fp)[[x]]^q\big).
\end{array}
$$
\end{enumerate}
Note that  $\hat R_\fp = A_\fp [[x]]=(R_\fp)^\wedge \cong (R_{\fn_\fp})^\wedge$ and
$\hat R(\fp) = k(\fp)[[x]]=R(\fp)^\wedge \cong (R_{\fn_\fp}/\fp R_{\fn_\fp})^\wedge$ are local rings but  $R_\fp \not\cong R_{\fn_\fp}$ and
 $R(\fp) \not\cong R_{\fn_\fp}/\fp R_{\fn_\fp}$, since  $R_\fp$ and $R(\fp)$ are in general not local. 
\end{lemma}

\begin{proof}  1. 
The natural inclusion $R_{\fn_\fp} = A[[x]]_{\fn_\fp} \xhookrightarrow {\text{} }  A_{\fp} [[x]] $ is given as follows. 
Let $h/g \in R_{\fn_\fp}$ with $h, g \in R$,  $g \notin \fn_\fp$ and
 write $g =g_0 - g_1$ with $g_0 \in A$ and $g_1 \in \x R$.  Then $g \notin \fn_\fp = \langle \fp, x\rangle$ iff $g_0 \notin \fp$ and $g$ is a unit in  $R_{\fn_\fp}$ iff its image in $A_{\fp} [[x]] $ is a unit.
 We get
$$h/g = \frac{g_0^{-1}h}  {(1-g_1/g_0)} = g_0^{-1}h\sum_{i \ge 0} (g_1/g_0)^i \in A_{\fp} [[x]].$$
Now it is not difficult to see that the induced map 
$A[[x]]_{\fn_\fp}/\x ^m   \to  A_{\fp} [[x]] /\x ^m$ is bijective
(a finite sum  $\sum_{|\alpha|= 0}^{m-1} (a_\alpha/b_\alpha)x^\alpha,  a_\alpha, b_\alpha \in A,  b_\alpha \notin \fp$ in $A_{\fp} [[x]]$ can be written as
 $1/b \sum_{|\alpha|= 0}^{m-1} (a_\alpha b'_\alpha)x^\alpha$ with $b = \prod b_\alpha \notin \fn_\fp$, $b'_\alpha= b/b_\alpha \in A$, and hence is in $ A [[x]]_{\fn_\fp}$).
We get
$$ R_{\fn_\fp}\hat\otimes_A A = \lim\limits_{\longleftarrow} A[[x]]_{\fn_\fp}/\x ^m
\otimes_A A
= \lim\limits_{\longleftarrow}A_{\fp} [[x]] /\x ^m = A_{\fp} [[x]]$$
and also 
$R_{\fn_\fp}\hat\otimes_{A _\fp}A_\fp=A_{\fp} [[x]]= R_{\fn_\fp}^\wedge.$
Now apply Corollary \ref{cor.rex} to the presentation of $M$ and deduce the claim for $M\hat\otimes_A A_\fp$.

2. $\hat M(\fp)  = M\hat\otimes_A (A_{\fp} / \fp A_{\fp})= 
(M\hat\otimes_A A_{\fp}) / \fp (M\hat\otimes_A A_{\fp})=
M_{\fn_\fp}^\wedge/ \fp M_{\fn_\fp}^\wedge$
by Corollary \ref{cor:cpl} (iv) and  the first statement of this lemma.
Since $M_{\fn_\fp}$ is finitely presented over $R_{\fn_\fp}$ we have
$ M_{\fn_\fp}^\wedge= M_{\fn_\fp} \otimes_{R_{\fn_\fp}} R_{\fn_\fp}^\wedge$,
which implies the result.

3. This follows from Corollary \ref{cor.pres}.
\end{proof}

Over maximal ideals the fibre and the completed fibre coincide:

\begin{lemma}\label{lem.maxA}
Let $A$ be Noetherian and $M$ a finitely generated $R$-module. 
 For $\fa \subset A$ a {\em maximal} ideal the following holds.
\begin{enumerate}
\item [(i)] $$  \hat M(\fa) =  M(\fa), \ \  \hat d_\fa (M) = d_\fa (M). $$

\item  [(ii)] $R/ \fa R = R(\fa) = \hat R(\fa)  =k(\fa)[[x]]$ and  $\fa R$ is a prime ideal in $R$. 

\item  [(iii)] $\fn_\fa $ is a maximal ideal of $R$ and any maximal ideal  of $R$ is of the form $\fn_\fa $ for some $\fa \in \Max A$.  Hence $\langle x\rangle$ is 
contained in the  Jacobson radical  of $R$.
\item  [(iv)] $M(\fa) = M/\fa M \cong  M_{\fn_\fa}/\fa M_{\fn_\fa}$.
\end{enumerate}
\end{lemma}

\begin{proof}
 (i) Since $\fa$ is maximal, $k(\fa) = A/\fa$ is a finite $A$-module. 
 Corollary \ref{cor:cpl} (iv) implies 
 $\hat M(\fa) = M \hat \otimes_A A/\fa  = M \otimes_A A/\fa= M(\fa).$\\
(ii) This follows from (i) and the fact that $R/\fa R =k(\fa)[[x]]$ is integral.\\
(iii) Cf.  \cite[\S 1, Example 1]{Mat86} and
\cite [Chapter 1, Exercise 5]{AM69}). \\
(iv) $ M/\fa M = M\otimes_A A/\fa = M(\fa)=\hat M(\fa) =
M \hat \otimes_A A/\fa =\coker (R_{\fn_\fa}^p\hat\otimes_A A/\fa \to R_{\fn_\fa}^q\hat\otimes_A A/\fa) = M_{\fn_\fa}/\fa M_{\fn_\fa}$, as in the proof of 
Lemma \ref{lem.loc}. 
\end{proof}

\medskip

As a first step to semicontinuity we show below (Lemma \ref{lem.open}) that the vanishing locus of
$\hat d_\fp(M)$ is open.
For an arbitrary $A$-module $M$ $$\Supp_A(M) := \{\fp \in \Spec A\ | \ M_\fp \neq 0\}$$  denotes the support of $M$  and 
$\Ann_A(M) = \{f \in A \, | \, fM=0 \}$ the annihilator ideal of $M$. 

In general  $\Supp_A (M)$ is not closed in $\Spec A$, but  if $M$ a finitely generated $A$-module, then it is well known that $\Supp_A(M) = V(\Ann_A(M))$, which is closed in $\Spec A$. More generally we have:

\begin{remark} \label{rem.ann}  {\em 
 For  any $A$-module $M$ we have 
 $$\Supp_A (M) \subset V(\Ann_A(M)).$$ 
 If $R$ is an $A$-algebra  and $M$ an $R$-module, then $\Ann_A(M) = \Ann_R(M) \cap A$. If 
  $M$  is a finite $R$-module then  
  $$\Supp_A(M) = V(\Ann_A(M))$$
and  hence $\Supp_A(M)$ is closed in $\Spec A$. 

To see the first claim,  let $\fp \in \Spec A$. Note that $\fp \notin \Supp_AM \Leftrightarrow M_\fp=0 \Leftrightarrow \forall m\in M \  \exists f \in A$, $f \! \notin \! \fp$,  $fm=0$ and that 
$\fp \notin V(\Ann_A(M)) \Leftrightarrow  \Ann_A(M) \nsubseteq \fp   \Leftrightarrow   \exists f \! \notin \! \fp, f M= 0 $. Hence $\fp \notin V(\Ann_A(M))$ implies $\fp \notin \Supp_A(M)$, i.e. $\Supp_A (M) \subset V(\Ann_A(M))$.

Now let $M$ be generated over $R$ by $m_1,\dots,m_k \in M$. If $M_\fp=0$, choose  $f_i \in A, f_i \! \notin \! \fp,  f_im_i=0.$ Then 
$f=f_1\cdots f_k \! \notin \! \fp$ satisfies $fM=0 $ and hence $f\in \Ann_A(M)$. We get $\fp \notin V(\Ann_A(M))$ and hence the other inclusion $\Supp_A (M) \supset V(\Ann_A(M))$.
}
\end{remark}

In our situation for $R=A[[x]]$ and $M$ finitely $R$-presented we have (possibly strict) inclusions (cf. Lemma \ref{lem.cf})
$$\{ \fp \in \Spec A \ | \ \hat d_\fp(M) \neq 0\} \subset \{\fp \,|\,d_\fp(M)\neq0\}  \subset \Supp_A (M),$$
where the first (Lemma \ref{lem.open}) and the last (Remark \ref{rem.ann}) sets are closed in $\Spec A$, while the middle set may not be closed (Example \ref{ex.Kt}.4).

\begin{lemma}\label{lem.open} Let $M$ be a finitely presented $R=A[[x]]$-module.  
\begin{enumerate}
\item We have (possibly strict) inclusions 
$$\{ \fp \in \Spec A \ | \ \hat d_\fp(M) \neq 0\} \subset \{\fp \,|\,d_\fp(M)\neq0\}  \subset \Supp_A (M),$$
where the first and the last  sets are closed in $\Spec A$, while the middle set may not be closed.
\item The map $\Supp_R (M) \to \Supp_A (M) ,  \fn \mapsto \fn \cap A,$ is dominant and
induces a  homeomorphism
$$V(\x) \cap \Supp_{R} (M) \xrightarrow{\text {$\approx$}} 
\{ \fp \in \Spec A \ | \ \hat d_\fp (M) \neq 0\}.$$
 Hence $\{ \fp \in \Spec A \ | \ \hat d_\fp (M) =0\}$, the {\em vanishing locus} of $\hat d_\fp (M)$, is open in $\Spec A$. 
\item  Let $A'=A/\Ann_A(M)$, $R'=A'[[x]]$ and denote by $M'$ the module $M$ considered as $A'$-module. Then $M'$ is a finitely presented $R'$-module and for 
$\fp \in \Spec(A') \subset  \Spec(A)$ we have $M_\fp = M'_\fp$, 
$M(\fp) =M'(\fp)$, $\hat M_\fp =\hat M'_\fp$, and $\hat M(\fp) =\hat M'(\fp)$. 
For 
$\fp \in \Spec(A) \setminus  \Spec(A')$ the modules $M_\fp,  
M(\fp), \hat M_\fp$, and $\hat M(\fp)$ vanish.

\noindent 
{\em In particular, we may consider $M$ as an $A'$-module, whenever we study the fibres or the completed fibres of $M$.}
 \end{enumerate}
\end{lemma}

\begin{proof} 1. The first inclusion follows from Lemma \ref{lem.cf}(i), the second from the definition. For an example where these inclusions are strict, see
Example \ref{ex.Kt}.2, 3 and \ref{ex.Kt}.4. The first set is closed by item 2. and the third by Remark \ref{rem.ann}. In Example \ref{ex.Kt}.4 the middle set is not closed.

2. Since $\Ann_A(M) = \Ann_R(M) \cap A $, the map $ A/ \Ann_A(M)
\to R/\Ann_R(M)$ is injective and induces therefore a dominant morphism 
$ \Spec (R/\Ann_R(M)) \to  \Spec (A/\Ann_A(M))$. The first claim follows hence from Remark \ref{rem.ann}.

For the second claim consider  $ A[[x]]^p \xrightarrow{\text {$T$}} A[[x]]^q \to  M \to 0$, a  presentation of $M$ with $T=(t_{ij}), \, t_{ij} \in A[[x]]$, and let $\fp \in  \Spec A$. Then 
$ k(\fp)[[x]]^p \xrightarrow{\text {$T'$}} k(\fp)[[x]]^q \to  \hat M(\fp) \to 0$
is a presentation of $\hat M(\fp)$ with $T'=(t'_{ij}), \,t'_{ij} \in k(\fp)[[x]]$, the induced map  (Corollary \ref{cor.pres}). 

Now $\hat M(\fp) = 0$ iff $T'$ is surjective, i.e., iff the 0-th Fitting ideal (the ideal of $q$-minors) of $T'$  contains a unit $u'  \in k(\fp)[[x]]$. Write $u'$ as
$u' = u'_0+ u'_1$  with
$u'_0 \in k(\fp)\smallsetminus \{0\}, u'_1 \in \x k(\fp)[[x]]$. Since Fitting ideals are compatible with base change, the
$0$-th Fitting ideal $F_0 \subset A[[x]]$  of $M$ contains an element $u=u_0 + u_1 \in A[[x]]$ with
$u_0 \in A, u_1 \in \x A[[x]]$,
that maps to $u'$ under $A[[x]] \to A_\fp[[x]] \to A_\fp/\fp A_\fp[[x]]$. Hence $u'$ is a unit iff $u_0\notin \fp$,
i.e., iff $\langle x,\fp\rangle \notin V(F_0+\x)$. The result follows since 
$\Supp_R (M) = V(F_0 )$.

3. Any $R$-presentation of $M$ induces obviously an $R'$-presentation of $M'$.
The equalities $M_\fp = M'_\fp$ and  $M(\fp) =M'(\fp)$ for $\fp \in \Spec(A')$ are clear, the equalities 
$\hat M_\fp =\hat M'_\fp$ and $\hat M(\fp) =\hat M'(\fp)$ follow from this and from Lemma \ref{lem.cf}(i). Since $\Supp_A(M)=\Spec(A')$ by Remark \ref{rem.ann}, 
$M_\fp = M(\fp)=0$ for $\fp \notin \Spec(A')$ and Lemma \ref{lem.cf}(i) implies then
$\hat M_\fp = \hat M(\fp)=0$. 
\end{proof}  

At the end of this section we give two examples where $d_\fp(M)$ is not semicontinuous on $\Spec A$ while $\hat d_\fp(M)$ is. The examples show also that $\hat M(\fp)=0$ may happen for $ M(\fp) \ne 0$. 

\begin{example} \label{ex.Kt} {\em
Let $A=K[t]$, $K$ an algebraically closed field, $R=A[[x]]$, and 
$M = R/\langle t-x\rangle \cong K[[t]]$ as $A$-module via   $f(x,t) \mapsto f(t,t)$, with $t$ and $x$ one variable each. The following properties illustrate the difference between the fibres and the completed fibres. Let $\fa =\langle t-c\rangle, c \in K$, denote the maximal ideals in $A$.
By Lemma \ref{lem.maxA} $\hat M( \fa)  =  M(\fa) = M/ \fa M$ which is isomorphic to $K[[t]]/ \langle t-c\rangle.$ Hence $M(\langle t\rangle ) = K$ and  $M(\langle t-c\rangle ) = 0$ for  $c\neq 0$.

\begin{enumerate}

\item $M$ is not finitely generated over $A$, $d_\fa (M) = \hat d_\fa (M) <\infty$ for
$\fa \in \Max A$ and $d_\fa (M)$ is semicontinuous on $ \Max A$:\\
$K[[t]]$  is not finite  over $K[t]$ and $d_{\fa} (M) =1$ if  $c=0$
and  $0$ if $c\neq 0$.
\item $d_\fp (M)$ is not semicontinuous on $\Spec A$:\\
The prime ideal $\langle 0 \rangle$ is contained in every neighbourhood of  $ \fa =\langle t\rangle$ in  $\Spec A$. 
It satisfies $k(\langle 0 \rangle) = K(t)$ and we get
$M(\langle 0 \rangle) \cong K[[t]]\otimes_{A}K(t) = K((t))$, the field of formal Laurent series. Since $\dim_{K(t)} K((t)) = \infty$,  $d_{\langle 0 \rangle}(M) = \infty$, while  $d_{\fa} (M) \leq 1$ for $\fa \in \Spec A \smallsetminus \langle 0 \rangle$.
\item $\hat d_\fp (M)$ is semicontinuous on $\Spec A$: \\
We have  $\hat M(\langle 0\rangle)  =
K(t)[[x]]/\langle t-x\rangle$ by Corollary \ref{cor.pres}. Since $t$ is a unit
in $K(t)$,  $\hat d_{\langle 0\rangle}(M) = 0$.
\item $M({\fa}) = M_{\fa} /\fa M_{\fa} = 0$ does not imply $M_{\fa} =0$:\\
In fact, we have $M_{\langle t-c\rangle} \cong K[[t]]_{\langle t-c\rangle}$ as a $K[t]$-module. For $c \neq0$ we get  $ K[[t]]_{\langle t-c\rangle} = K((t))$ (since $t \notin \langle t-c\rangle$) while 
$K[[t]]_{\langle t-c\rangle}/ {\langle t-c\rangle}K[[t]]_{\langle t-c\rangle}=
K((t)) / {\langle t-c\rangle} = 0$.
We have $ \{\fp \,|\,\hat d_\fp(M)\neq0\} = \{\langle t \rangle\} \subsetneqq
\{\fp \,|\,d_\fp(M)\neq0\}= \{\langle t \rangle , \langle 0 \rangle\}  \subsetneqq \Supp_A (M) = \Spec A$.

\item $M$ is flat over $A$. By 1. and 3. we cannot expect any continuity of $d_\fp (M)$ 
or $\hat d_\fp (M)$ on $\Max A$ or on $\Spec A$ for flat $A$-modules.
\item The quasi-finite locus of $A \to A[[x]]/\langle t-x\rangle$ is not open in $\Spec A$:\\
The quasi-finite locus $\{\fp \in \Spec A \, | \,  d_\fp (M) < \infty \}$ is $\Spec A \smallsetminus \langle 0 \rangle$ by 1. and 2. Recall that if $B$ is a ring of finite type over $A$, then the quasi-finite locus of $A \to B$ is open by Zariski's main Theorem (cf. \cite[10.122]{Sta19}).
\item The quasi-completed-finite locus of $A \to A[[x]]/\langle t-x\rangle$ is open in $\Spec A$:\\
 Let us call $\{\fp \in \Spec A \, | \,  \hat d_\fp (M) < \infty \}$ the  {\em quasi-completed-finite locus}. It is $\Spec A$ in our example. \\
  In general, if semicontinuity of $\hat d_\fp (M)$ holds (Corollary \ref{cor.qf}), then the quasi-completed-finite locus is open. 
\end{enumerate} 
}
\end{example} 

\begin{example} \label{ex.Z} {\em
The following example may be of interest for arithmetic and computational purposes. It goes along similar lines as Example \ref{ex.Kt}.

Let $A=\Z$, $R=\Z[[x]]$, and $M = R/\langle x-p\rangle$, $p \in \Z$ a prime number. Since $R= \lim\limits_{\longleftarrow}\Z[x]/\langle x \rangle^n $ we obtain
$M = \lim\limits_{\longleftarrow}\Z/p^n = \hat \Z_{\langle p\rangle}$, the ring of 
$p$-adic integers.

Now let $\langle q\rangle \in \Max \Z$.\\ 
If $q\ne p$ then $q$ is a unit in 
$\Z_{\langle p\rangle}$ hence in $\hat \Z_{\langle p\rangle}$ and 
$M\otimes_\Z \Z/q  = M/\langle q\rangle M  =\hat \Z_{\langle p\rangle}/q \hat \Z_{\langle p\rangle} =0.$ \\
If $q = p$ then $M\otimes_\Z \Z/p = \hat \Z_{\langle p\rangle}/p \hat \Z_{\langle p\rangle} =  \Z/p.$ \\
Hence $d_{\langle q\rangle}(M) = \dim_{\Z/q}M\otimes_\Z \Z/q$ is  0 if  $q\ne p$ and 1 if  $q= p$. 

On the other hand, looking at the prime ideal $\langle 0\rangle$ we get 
$$\hat M(\langle 0\rangle)= M\hat \otimes_\Z \Q =\Q[[x]]/\langle x-p\rangle=0,$$
while
$$M(\langle 0\rangle)= M\otimes_\Z \Q = \hat \Z_{\langle p\rangle}\otimes_\Z \Q =Quot(\hat \Z_{\langle p\rangle}) $$ 
has dimension $d_{\langle 0\rangle}(M) = \dim_{\Q} Quot( \hat \Z_{\langle p\rangle}) =  \infty$.

To see the last equality in the formula for $M(\langle 0\rangle)$ one checks that the following diagram has the universal property of the tensor product:
\[
\begin{xymatrix}{
   \hat \Z_{\langle p\rangle} \ar[r]^-{i_1} & Quot( \hat \Z_{\langle p\rangle})=(\hat \Z_{\langle p\rangle})_p  \\
   \Z \ar[r]_{j_2} \ar[u]^{i_2}  & \Q   \ar[u]_{j_1}& \hspace{-1.7cm} .}
   \end{xymatrix}
\]
Here $i_1, i_2$ and $j_2$ are the canonical inclusions and $j_1$ is given as follows: if
$\alpha, \beta \in \Z$, $p \nmid \beta$, then 
$j_1(\frac{\alpha}{p^m\beta}) = \frac{1}{p^m}\frac{\alpha}{\beta}$, 
$\frac{\alpha}{\beta} \in \hat \Z_{\langle p\rangle}$ since $p \nmid \beta$.
The universality of the diagram is easily seen. If $T$ is a $\Z-algebra$ and $\phi:  \hat \Z_{\langle p\rangle} \to T$ and 
$\psi: \Q \to \Z$ are $\Z$-algebra homomorphisms then the morphism 
$\sigma: (\hat \Z_{\langle p\rangle})_p \to T$, given as 
$\sigma(\alpha/p^m) = \phi(\alpha)\psi(1/p^m), p \nmid \alpha$ is the unique one, making the obvious diagram commutative.
}
\end{example} 
\medskip

\subsection{Semicontinuity over a 1-dimensional ring}\label{ssec.1dim}

In this section $A$ will be Noetherian and $M$ a finitely generated $R$-module
(unless we say otherwise). Then $R = A[[x]]$ is Noetherian and $M$ is finitely presented as $R$-module.
At the moment we can prove the semicontinuity of $d_\fq(M)$ on $\Spec A$ 
in full generality only under certain assumptions on the irreducible components of $\Supp_R(M)$. This includes the case $\dim A =1$ where $\dim A$ denotes the Krull dimension of $A$. 
The case  of arbitrary Noetherain $A$ is treated in the next section under the assumption that the presentation matrix of $M$ is algebraic. 
\medskip

In an important special case semicontinuity holds for arbitrary $A$:

\begin{proposition} \label{prop.ann}
Let $A$ be Noetherian and $M$ a finitely generated $R$-module. 
\begin{enumerate}
\item If \, $\Supp_R (M) \subset V(\x)$ then  $M$ is finitely generated  over $A$ and
$\hat M(\fq) = M(\fq)$ for all $\fq \in \Spec A$. In particular 
semicontinuity of $\hat d_\fp(M)= d_\fp(M)$ holds at 
 any  $\fp \in \Spec A$.
 \item
If $M$ is finitely generated as an $A$-module (in particular $d_\fp(M) < \infty$ for  $\fp \in \Spec A$),
 then  $\hat d_\fp(M)\leq d_\fp(M)$ and  $\hat d_\fp(M)$ (as well as $d_\fp(M)$) is semicontinuous at  any  $\fp \in \Spec A$.
 \end{enumerate}
\end{proposition}

\begin{proof}
1. Since  $V(\Ann_R(M)) = \Supp_R (M)  \subset V(\x)$, we have 
$\x \subset \sqrt {\Ann_R(M)}$  and 
there exists an $m$ such that  $\x^m \subset \Ann_R(M)$. 
We get a surjection 
$$A[[x]]/\langle x \rangle^m  \to R/\Ann_R(M)$$
and since $A[[x]]/\langle x \rangle^m$ is finitely generated over $A$  this holds for $R/\Ann_R(M)$ too.  
Since  $M$ is finitely generated over $R/\Ann_R(M)$ it is finitely generated over $A$ and hence  finitely presented. 
By Lemma \ref{lem.Afin} there there is an open neighborhood $U$ of $\fp$ in 
$\Spec A$ such that 
	$d_\fq (M) \le d_\fp (M),  \ \fq \in U.$	
By Proposition \ref{prop.pctp}.5, $\hat M(\fq) = M(\fq)$ for all $\fq \in \Spec A$, showing the claim.

2.
Let $m<m'$ be two strictly positive integers and consider the natural surjective maps
\begin{center}
$R \longrightarrow R/\langle x\rangle^{m'} \longrightarrow R/\langle x \rangle^m$.
\end{center}
By the right exactness of the tensor product, they induce surjective maps
\begin{center}
$M \longrightarrow M/\langle x\rangle^{m'}M \longrightarrow M/\langle x \rangle^mM$.
\end{center}
and 
\begin{center}
$M(\fp)=M\otimes_A k(\fp)\longrightarrow M/\langle x\rangle^{m'}M \otimes_Ak(\fp)\xrightarrow{\pi_{m,m'}(\fp)} M/\langle x \rangle^mM \otimes_Ak(\fp)$.
\end{center}
Since $M$ is finitely generated over $A$, all the modules appearing in the last sequence are finite-dimensional $k(\fp)$-vector spaces, with $\dim_{k(\fp)}(M(\fp))=d_{\fp}(M)$ by definition. Since $\pi_{m,m'}(\fp)$ is surjective for all $m<m'$, the elements of the inverse system $\{M/\langle x \rangle^mM\otimes_Ak(\fp)\}_m$ of finite-dimensional $k(\fp)$-vector spaces have dimensions increasing with $m$ and bounded above by $d_{\fp}(M)$. Thus $\dim_{k(\fp)}\big(M/\langle x \rangle^mM\otimes_Ak(\fp)\big)$ stabilizes for large $m$. Hence, for $m$ large
\begin{center}
$M/\langle x \rangle^mM\otimes_Ak(\fp)=\underset{m'}\varprojlim(M/\langle x \rangle^{m'}M\otimes_Ak(\fp))=M\hat\otimes_Ak(\fp)=\hat M(\fp)$.
\end{center}
This implies $\hat d_{\fp}(M)=\dim_{k(\fp)}\hat M(\fp) \leq d_{\fp}(M)$. 

To see the semicontinuity of $\hat d_\fp(M)$, we use the semicontinuity of 
$d_{\fp}(M)$ (by Lemma \ref{lem.Afin}). It follows that there exists an open neighbourhood $U$ of $\fp$ such that the sequence $\{\dim_{k(\fq)}\big(M/\langle x \rangle^mM\otimes_Ak(\fq)\big)\}_m$ is  bounded by $d_{\fp}(M)$ simultaneously 
for all $\fq \in U$.
Hence,
$$d_\fq(M/\x^mM) = \dim_{k(\fq)} (M/\langle x \rangle^mM\otimes_Ak(\fq))=\dim_{k(\fq)}(\hat M(\fq))=\hat d_{\fq}(M)$$
 for a fixed  large  $m$ and $\fq \in U$. Since $M/\x^mM$ is finitely $A$-generated,
 Lemma \ref{lem.Afin} implies that $d_\fp(M/\x^mM)$ is semicontinuous,  and so is $\hat d_{\fp}(M)$.
\end{proof}

The inequality $\hat d_\fp(M)\leq d_\fp(M)$ of item 2. and its proof, as well as Example \ref{ex.referee}, were suggested to us by the anonymous referee.
 Note that $\hat d_\fp(M) = d_\fp(M)$ for $\fp$ a maximal ideal (by Lemma \ref{lem.maxA}), but that
$\hat d_\fp(M) < d_\fp(M)$ may happen by Example \ref{ex.complete} for $\fp$ not maximal.\\

Before we formulate the next result, we introduce some notations to  be used throughout this section.
Consider a minimal  primary decomposition of $\Ann_R(M)$,
	\[ \Ann_R(M)=\mathop\cap\limits_{i=1}^rQ_i\subset R.\]
Since $M$ is finitely generated over $R$,  $\Supp_R (M) = V(\Ann_R(M))= \mathop\cup\limits_{i=1}^r V(Q_i)$  and 
$\dim M = \dim \Supp_R (M)$. 

Let $P_1, \ldots P_s \subset R$ be the minimal associated primes of $\langle x \rangle$. Since they correspond via $h: A \cong R/\x$
to the minimal associated primes  $\bar P_1, \ldots \bar P_s$ of $A$, we have  $\dim V(P_j) \leq \dim A.$ 

\medskip

\begin{lemma} \label{lem.assprimes}
For $A$ Noetherian, $M$  finitely generated over  $R$ and $\fp \in \Spec A$ the following holds:

\begin{enumerate}
\item Let $A'$ be the reduction of $A$, $R' = A'[[x]]$ and $M'$ the $R'$-module
$ M\otimes_R R'$. Then $\hat M'(\fp)  \cong \hat M(\fp)$ and hence
$\hat d_\fp(M') = \hat d_\fp(M)$.

\item Let $Q\subset R$ be an ideal. Then 
$\hat d_\fp(M/QM) \leq \hat d_\fp(M)$.

\item  If  $Q_i \not\subset \fn_\fp$ for some $1 \le i \le r$, then  
$\hat d_\fq(M)  =\hat d_\fq(M/QM)$, with $Q= \mathop\cap_{j\ne i}Q_j$, for $\fq$ in some neighbourhood of $\fp$ in $\Spec A$.

\item If  $Q_i \subset \fn_\fp$ and $\dim V(Q_i) > \dim (A/Q_i\cap A)_\fp$ for some $1 \le i \le r$, then $\hat d_\fp(M) = \infty$.

\item Let $U=\Spec B \subset \Spec A$ be an affine open neighbourhood of $\fp$ and $M_B = M\otimes_A B$ the restriction of $M$ to $U$. Then 
$\hat M_B(\fq) = \hat M(\fq)$ for all $\fq \in U$.
\end{enumerate}
\end{lemma}

\begin{proof}
1. Since $\fp \in \Spec A$ contains the nilpotent elements, $A'/\fp'  = A/\fp$,
where $\fp'$ is the image of $\fp$ in $A'$, and hence the residue field does not change if we pass from  $A$ to $A'$. By Proposition \ref{prop.pctp}.1 we have
$\hat R'(\fp') = R'\hat\otimes_{A' } k(\fp) =  k(\fp)[[x]] = \hat R(\fp). $ 
Consider a presentation
$ R^p \xrightarrow{\text {$T$}} R^q \to  M \to 0$
of $M$. Applying $\otimes_R R'$, we get a presentation of $M'$,
$ R'^p \xrightarrow{\text {$T$}} R'^q \to  M' \to 0.$
Apply $\hat\otimes_A k(\fp)$ to the first resp.  $\hat\otimes_{A'} k(\fp)$ to the second exact sequence above. The sequences stay exact by Corollary \ref{cor.rex}. Since $(\hat R(\fp))^{k} = (\hat R'(\fp'))^{k}$
it follows that  the canonical morphism $M \to M'$ induces  an isomorphism  
$\hat M(\fp) \cong \hat M'(\fp)$.
\medskip

2.  Since  $(M/QM)\hat\otimes_A k(\fp) = M\hat\otimes_A k(\fp) / Q(M\hat\otimes_A k(\fp))$  by Corollary \ref{cor.rex}, the result follows.
\medskip

3. $Q_i \not\subset \fn_\fp$ means that $ \fn_\fp$ is not a point of $V(Q_i)$. Hence  $ \fn_\fq \notin V(Q_i)$ and
$M_{\fn_\fq} = (M/Q)_{\fn_\fq} $ for $ \fn_\fq$  in some neighbourhood of $ \fn_\fp$ in 
$V(\x)$. The result follows from Lemma \ref{lem.loc}. 
\medskip

4.   Set $\bar R := R/Q_i$, $\bar A:= A/Q_i\cap A$ and $\bar M := M/Q_iM$. Then  $Q_i \subset \Ann_R(\bar M)\subset \sqrt{(Q_i)}$
and  $\dim \bar  R_{\fn_\fp}  = \dim \bar M_{\fn_\fp} 
=\dim V(Q_i) > \dim \bar A_\fp$ by assumption. 
Considering  $\bar M$ as $R$-module, we have 
$  \bar M\hat\otimes_A k(\fp) = \bar M(\fp)^\wedge =  (\bar M_{\fn_\fp} / \fp \bar M_{\fn_\fp})^\wedge$ by Lemma \ref{lem.loc}. 
Since the  $\bar R_{\fn_\fp}$-modules  $\bar M(\fp)$ and its $\x$-adic completion $\bar M(\fp)^\wedge$ have the same Hilbert-Samuel function w.r.t. 
$\fn_\fp$, their dimension coincides (c.f. \cite[Corollary 5.6.6]{GP08}).
Moreover,
$\fp \bar R_{\fn_\fp} $ is the annihilator of $\bar M_{\fn_\fp}/\fp \bar M_{\fn_\fp}$ and therefore 
$\dim  \bar M(\fp)^\wedge = \dim \bar R_{\fn_\fp}/\fp \bar R_{\fn_\fp}.$

We apply now \cite[Theorem 15.1]{Mat86} to the map of local rings 
$\bar A_\fp \to \bar R_{\fn_\fp}$ and get that  
$\dim \bar R_{\fn_\fp}/\fp \bar R_{\fn_\fp} \ge \dim \bar R_{\fn_\fp} - \dim \bar A_\fp > 0$ and hence 
$\dim_{k(\fp)}  \bar M(\fp)^\wedge  = \infty$.
Then $\hat d_\fp(M) = \dim_{k(\fp)} M(\fp)^\wedge = \infty$ by 2. of this lemma. 
\medskip

5. We may assume that $B=A_f$ for some $f \notin \fp$. 
Since $A_\fq = (A_f)_\fq$ for $\fq \in U = D(f)$, we have $k(\fq) =A_f\otimes_Ak(\fq)$. Now 
Proposition \ref{prop.pctp}.1 implies
$\hat M_{A_f}(\fq) = (M\otimes_A A_f)\hat\otimes_Ak(\fq) = (M\otimes_A A_f\otimes_Ak(\fq))^\wedge = M(\fq)^\wedge = \hat M(\fq).$
 \end{proof}
 \smallskip
 
 \begin{proposition} \label{prop.power}
Let $A$ be Noetherian, $M$ a finitely generated $R$-module and  
fix $\fp \in \Spec A$.  Let $Q_1,\ldots, Q_r$ be the primary components of $Ann_R (M)$, which we renumerate such that  
\begin{enumerate}
\item[I.]  $V(Q_i) \subset V(\x)$ for $1 \leq i \leq k$,  
\item[II.]   $V(Q_i) \not\subset V(\x)$ for $k+1 \leq i \leq r,$ 
\end{enumerate}
and set 
$Q_I  := \mathop\cap\limits_{i=1}^kQ_i$ and $Q_{II}:=\mathop\cap\limits_{i=k+1}^rQ_i.$ 
Assume  that either (a) $V(Q_{II}) = \emptyset$ (i.e., $k=r$), or (b) $\dim V(Q_{II}) > \dim V(Q_{II}\cap A)$ or (c) $\fn_\fp$ is an isolated point of $V(\x)\cap V(Q_{II})$.

Then there is an open neighbourhood $U$ of $\fp$  in $\Spec A$ such that  $\hat d_\fq(M) \leq \hat d_\fp(M)$ for all prime ideals $\fq \in U.$
\end{proposition}

\begin{proof}
We set $M_I := M/Q_IM$ and $M_{II}  := M/Q_{II}M.$ Then $\Ann_R(M_I) = Q_I$ and $\Ann_R(M_{II}) = Q_{II}$. 
By  Lemma \ref{lem.assprimes}.3 we may assume that 
 $Q_i \subset \fn_\fp$ for all $1 \le i \le r$. 

We have $\Supp_R(M_I) = V(Q_I) \subset V(\x)$.
By Proposition \ref{prop.ann} there is an open neighborhood $U_1$ of $\fp$ in 
$\Spec A$ such that 
	\begin{align*}
	\hat d_\fq (M_I) \le \hat d_\fp (M_I),  \ \fq \in U_1. \tag{1}\label{1} 
	\end{align*}
	
(a) If $V(Q_{II})=\emptyset$, then $M=M_I$ and the claim follows from  (\ref{1}). 
\smallskip

(b) If  $\dim V(Q_{II}) > \dim V(Q_{II}\cap A)$ then $\dim V(Q_i) > \dim (A/Q_i\cap A)_\fp$ for some $i$ and hence $\hat d_\fp(M) = \infty$ by Lemma \ref{lem.assprimes}.4, implying the claim.
\smallskip

%

(c) Now let  $\fn_\fp$ be an isolated point of $V(\x)\cap V(Q_{II})$. Then 
there exists an open neighbourhood  $U_2 \subset \Spec A$ of $\fp$  
such that $M_{I,{\fn_\fq}} = M_{\fn_\fq}$ if $\fq \in U_2 \setminus \{\fp\}$. Since $\hat d_\fq (M)=  \dim_\k (M_{\fn_\fq}/\fq M_{\fn_\fq})^\wedge$ we get
	\begin{align*}
	\hat d_\fq (M) = \hat d_\fq (M_I), \ \fq \in \ U_2 \setminus \{\fp\} \tag{2}\label{2}.
	\end{align*}
Using  (\ref{1}) and (\ref{2}), we have $\hat d_\fq (M) \leq \hat d_\fp (M_I)$  
for $\fq \in U_1\cap U_2 \setminus \{\fp\}$ and
by Lemma  \ref{lem.assprimes}.2  
	\begin{align*}
	\hat d_\fp (M_I) \le \hat d_\fp (M) \tag{3}\label{3}.
	\end{align*}
Hence  $\hat d_\fq (M) \leq \hat d_\fp (M)$  
for $\fq \in U_1\cap U_2$. 
%
%
%
%
%
\end{proof}

As a corollary we get the following theorem, which was already proved for maximal ideals in $A = \k[t]$ in  \cite{GPh19}.

\begin{theorem} \label{thm.1dim}
Let $A$ be Noetherian, $M$ a finitely generated $R$-module and  $\fp \in \Spec A$. 
If $\dim_\fp(\Supp_A(M)) \leq 1$ then there is an open neighbourhood $U$ of $\fp$  in $\Spec A$ such that  
$$\hat d_\fq(M) \leq  \hat d_\fp(M) \ \text {for all } \fq \in U.$$
\end{theorem}

\begin{proof} By Lemma \ref{lem.open}.3 we may assume that $\Ann_A(M) =0$, such that $\Supp_A(M) = A$.
We may further assume that $\dim A = \dim V(\x) =1$ and $\hat d_\fp(M) < \infty$. Using the notations from Proposition \ref{prop.power}, we have $\dim V(Q_{II}\cap A) \leq 1$ and by the proof of  Proposition \ref{prop.power}(b) that  $\dim V(Q_{II} )\leq 1$.
Hence, either $ V(Q_{II}) = \emptyset$, or $\fn_\fp$ is an isolated point of $V(\x)\cap V(Q_{II})$.
The result follows from Proposition \ref{prop.power}.
\end{proof}

\begin{corollary}   \label{cor.Z}
Let  $A=\Z$ and let $M$ be a finitely generated  $\Z[[x]]$-module, $x=(x_1\cdots x_n)$, given by a presentation 
$$ \Z[[x]]^r \to \Z[[x]]^s \to  M \to 0.$$
Denote by
$$M_p := \hat M({\langle p \rangle}) = \coker \big(\F_p[[x]]^r \xrightarrow{\text {$\bar T$}} \F_p[[x]]^s \big)$$ 
if $p\in \Z$ is a prime number and by
 $$M_0 := \hat M({\langle 0 \rangle}) = \coker \big (\Q[[x]]^r \xrightarrow{\text {$\bar T$}} \Q[[x]]^s \big )$$ 
 the induced modules.
\begin{enumerate}
\item  Fix a prime number $p$. If $\dim_{\F_p} M_p < \infty$ then
 $\dim_{\F_p} M_p \geq \dim_{\Q} M_0$. Moreover, for all except finitely many prime numbers $q\in \Z$,  $\dim_{\F_p} M_p \geq \dim_{\F_q} M_q$.

\item If $ \dim_{\Q} M_0  < \infty$ then 
$\dim_{\Q} M_0 \geq \dim_{\F_q} M_q$ for all except finitely many prime numbers $q\in \Z$,
and hence ``='' for all except finitely many prime numbers $q\in \Z$.
\end{enumerate}
\end{corollary}

\noindent The first part of statement 1. follows, since $\langle0\rangle$ is in every neighbourhood of $p$.
In particular $\dim_{\Q} M_0$ is finite if $\dim_{\F_p} M_p$  is finite for some prime number $p$.

\begin{remark}\label{rem.Z}
{\em 
The corollary is important for practical computations in computer algebra systems. For simplicity let $I$ be an ideal in $\Z[[x]]$ generated by polynomials, $M= \Z[[x]]/I$, and $I_p$ the image of $I$ in $\F_p[[x]]$.
The dimension of $\Q[[x]]/I$ resp. of $\F_p[[x]]/ I_p$, if finite, is equal to the dimension of
$\Q[x]_{\x}/I$ resp. of $\F_p[x]_{\x}/I_p$. These dimensions 
can be computed in the localizations $\Q[x]_{\x}$ resp.  $\F_p[x]_{\x}$ by computing a Gr\"obner or standard basis of $I$ resp. of $I_p$ w.r.t. a local monomial ordering (cf. \cite{GP08}). Such algorithms are implemented e.g. in {\sc Singular} \cite{DGPS}.
Usually the computations over $\Q$ are very time consuming or do not finish, due to extreme coefficient growths, and therefore often modular methods are used. The above corollary says that for all except finally many prime numbers $p$ we have equality 
$\dim_\Q \Q[x]_{\x}/I = \dim_{\F_p} \F_p[x]_{\x}/I_p$, and if this holds $p$ is
sometimes called a ``lucky'' prime.
This fact  can also be proved by Gr\"obner basis methods. More interesting is however
 that $\dim_\Q \Q[x]_{\x}/I < \infty$ if  there exists just one $p$ (lucky or not) such that $\dim_{\F_p} \F_p[x]_{\x}/I_p<\infty $ and that the first dimension is bounded by the latter. This was stated in \cite{Pf07} without proof.
}
\end{remark}
\medskip

\subsection {Henselian rings and Henselian tensor product}\label{ssec.hs}
\bigskip
In this section we  recall some basic facts about Henselian rings and introduce similarly to the complete tensor product a Henselian tensor product. For details about Henselian rings see  \cite{Sta19} or \cite{KPR75}. The  Henselian tensor product is needed in Section \ref{ssec.poly} for algebraically presented modules.
We start with some basic facts about  \'etale ring maps. 
\begin{definition} \label{def.etal}\noindent
\begin{enumerate}
\item
A ring map $\phi:A \longrightarrow B$ is called {\em \'etale} if it is flat, unramified  and of finite presentation.\footnote{$\phi$ is unramified if it is of finite type and the module of K\"ahler differentials $\Omega_{B/A}$ vanishes.  $\phi$  is of finite presantation if $B \cong A[x_1,...,x_n]/\langle f_1,...,f_m\rangle$ as $A$-algebras.}
\item $\phi$ is called {\em standard \'etale} if $B=(A[T]/F)_G$, $F,G \in A[T]$, the univariate polynomial ring, $F$ monic and $F'$ a unit in $B$.
\item $\phi$ is called {\em \'etale at $\fq \in \Spec(B)$} if there exist $g \in B \smallsetminus \fq$ such that $A\longrightarrow B_g$ is  \'etale.
\end{enumerate}
\end{definition}
The following proposition lists some basic properties of \'etale maps.
The results can be found in section 10.142 of \cite{Sta19}.
\begin{proposition}\label{prop.etal}\noindent
\begin{enumerate}
\item The map $A \longrightarrow A_f$ is \'etale.
\item A standard \'etale map is an \'etale map.
\item The composition of \'etale maps is \'etale.
\item A base change of \'etale maps is \'etale.
\item An \'etale map is open.
\item An \'etale map is quasi-finite.
\item Given $\phi:A \longrightarrow B$ and $g_1,\ldots,g_m \in B$ generating the unit ideal\,\footnote{ i.e., $\Spec(A)=\cup D(g_i)$.} such that $A \longrightarrow B_{g_i}$ is \'etale for all $i$ then  $\phi:A \longrightarrow B$ is \'etale.
\item Let $\phi:A \longrightarrow B$ be \'etale. Then there exist $g_1,\ldots,g_m \in B$ generating the unit ideal such that $A \longrightarrow B_{g_i}$ is standard \'etale for all $i$.
\item Let $S \subset A$ be a multiplicatively closed subset and assume that $\phi':S^{-1}A \longrightarrow B'$ is \'etale. Then there exists an  \'etale map $\phi:A \longrightarrow B$ such that $B'=S^{-1}B$ and $\phi ' = S^{-1}\phi$.
\item Let $\phi':A/I \longrightarrow B'$ be  \'etale for some ideal $I \subset A$. Then there exist an  \'etale map $\phi:A \longrightarrow B$ such that $B'=B/IB$ 
and the obvious diagram commutes.

\end{enumerate}
\end{proposition}

\begin{definition} \label{def.hens}\noindent
Let $A$ be a ring and $I\subset A$ an ideal. $A$ is called {\em Henselian with respect to $I$} if the following holds\,\footnote{Note, that (similarly to the $I$-adic completion) the condition implies that $I$ is contained in the Jacobson radical of $A$. If we start with an ideal contained in the Jacobson radical then it is enough to consider monic polynomials $F$ in the definition.} (Univariate Implicit Function Theorem):\\ 
Let $F\in A[T]$, the univariate polynomial ring,  such that $F(0)\in I$ and $F'(0)$ is a unit modulo $I$. Then there exists $a\in I$ such\,\footnote{Note that $a$ is uniquely determined by the condition $a\in I$, \cite{KPR75}.} that $F(a)=0$.
\end{definition}
Next we associate to any pair $(A,I)$, $I\subset A$ an ideal, the Henselization $A^h_I$, i.e. the "smallest"  Henselian ring with respect to $I$, such that $A^h_I \subset \hat A_I =\lim\limits_{\longleftarrow}\big(A/I^n \big)$ the $I$-adic completion.

\begin{definition} \label{def.henselization} \noindent
\begin{enumerate}
\item
Let $A$ be a ring and $I \subset A$ an ideal. The ring
$$A^h_I=\lim\limits_{\longrightarrow}\big(B \, | \, A\longrightarrow B \text{ an \'etale ring map inducing } A/I\cong B/IB \big)$$
is called the {\em Henselization of $A$ with respect to $I$}.
\item The Henselization of $A[x]$, $A$ any ring, $x=(x_1,\ldots,x_n)$, with respect to $I=\x$ is denoted by $A\x$. We call  $A\x$ the {\em ring of albegraic power series over $A$}.
\end{enumerate}
\end{definition}
The Henselization has the following properties (cf. section 15.11 and 15.12 of \cite{Sta19}):
\begin{proposition}\label{prop.hensel}
Let $A$ be ring and $I \subset A$ an ideal.
\begin{enumerate}
\item $A^h_I$ is Henselian with respect to $I^h=IA^h_I$ and $A/I^m=A^h_I/(I^h)^m$ for all $m$.
\item $A$ is Henselian with respect to $I$ if and only if $A=A^h_I$.
\item If $A$ is Noetherian then the canonical map $A \longrightarrow A^h_I$ is flat.
\item If $A$ is Noetherian then the canonical map $A^h_I \longrightarrow \hat A_I$ is faithfully flat and $\hat A_I$ is the $I^h$-adic completion 
of $A^h_I$.
\end{enumerate}
\end{proposition}
\begin{remark}\label{rem.hensel}{\em
The definition of the Henselization implies that $A^h_I$ is contained in the  algebraic closure of $A$ in  $\hat A_I$.
If $A$ is excellent\,\footnote{For the definition of excellence see 15.51 \cite{Sta19}).} then $A^h_I$ is the algebraic closure of $A$ in $\hat A_I$.
This is even true under milder conditions, see \cite{KPR75}.
In this situation $C\langle x \rangle$ is called the {\em ring of algebraic power series} of $C[[x]]$.}
\end{remark}

Next we prove a lemma which we need later in the applications.
\begin{lemma}\label{lem.hensel}
Let $A$ be a ring and $\fp \in \Spec(A)$ a prime ideal. Let $C=A_\fp$, $I=\fp C$ and $f_1,\ldots,f_m \in C^h_I$. 
Then there exists an \'etale map $A \longrightarrow B$
such that 
\begin{enumerate}
\item $f_1,\ldots,f_m \in B$,
\item there exists a prime ideal $\fq \in \Spec(B)$ such that $\fq \cap A =\fp$.
\end{enumerate}
\end{lemma}
\begin{proof}
By definition we have
$$C^h_I=\lim\limits_{\longrightarrow}\big(D\,| \,C\longrightarrow D \text{ \'etale inducing } C/I=D/ID \big).$$
We choose $D$ from the inductive system above such that $f_1,\ldots,f_m \in D$. Since $(C,I)$ is a local ring and
$C/I=D/ID$, the ideal $ID$ is a maximal ideal in $D$ and we have $ID \cap C=I$. Using Proposition \ref{prop.etal} (9)
for the multiplikatively closed system $S=A \smallsetminus \fp$ we find an \'etale map $A \longrightarrow B'$ such that
$D=S^{-1}B'$. This implies that $f_1,\ldots,f_m \in B'_g$ for a suitable $g \in S$. Let $B=B'_g$ and $\fq=ID \cap B$ then
$A \longrightarrow B$ is \'etale having the properties 1. and 2.
\end{proof}
Next we define the Henselization of an $A$-module $M$ with respect to an ideal $I \subset A$ similarly to the definition of the
Henselization of $A$ with respect to $I$.

\begin{definition} \label{def.hensModul}
Let $A$ be a ring, $I \subset A$ an ideal and $M$ an $A$-module. The module
$$M^h_I=\lim\limits_{\longrightarrow}\big(M \otimes_AB\,| \,A\longrightarrow B \text{ \'etale inducing } A/I=B/IB \big)$$
is called the {\em Henselization of $M$ with respect to $I$.}
\end{definition}
\begin{lemma}\label{lemma.hensModul}
$M^h_I=M\otimes_AA^h_I$.
\end{lemma}
\begin{proof}
The lemma follows since the direct limit commutes with the tensor product (cf. 10.75.2 \cite{Sta19}).
\end{proof}

\begin{definition}\label{def.htp}
Let $A$ be a ring, $R=A\langle x \rangle$, $B$ an $A$-algebra and $M$ an $R$-module. We define  the {\em henselian tensor product} of $R$  and $B$ over $A$ as
the ring
$$R \otimes^h_A B :=\lim\limits_{\longrightarrow}\big(C \,| \,B[x]\longrightarrow C \text{ \'etale inducing } B=C/\x C \big) =B\langle x \rangle\big).$$
$$M \otimes^h_A B :=\lim\limits_{\longrightarrow}\big(M\otimes_AC \,| \,B[x]\longrightarrow C \text{ \'etale inducing } B=C/\x C \big)=M\otimes_AB\langle x \rangle\big).$$
\end{definition}

The Henselian tensor product has similar properties as the complete tensor product. Especially we obtain the following lemma.
\begin{lemma}\label{lem.pres}
If  $ A\langle x \rangle^p \xrightarrow{\text {$T$}} A\langle x \rangle^q \to  M \to 0$ is an  $A\langle x \rangle$-presentation of $M$ then 
$$M\otimes^h_A B=\coker \big(B\langle x \rangle^p \xrightarrow{\text {$T$}}B\langle x \rangle^q\big).$$
In particular $R\otimes^h_A k(\fp) = k(\fp)\langle x \rangle$ for $\fp \in \Spec A$. 
\end{lemma}
\begin{definition} \label{def.ch}
Let $A$ be a ring, $R=A\langle x\rangle$ and $M$ an $R$-module. 
We define for $\fp \in \Spec A$ the $R\otimes^h_A k(\fp) = k(\fp)\langle x\rangle$-module
$$ M^h(\fp) := M \otimes^h_{A} k(\fp)$$
and call it the {\em Henselian fibre} of $M$ 
over $\fp$. Moreover, we set
$$d^h_\fp (M) := \dim_{k(\fp)} M^h(\fp).$$
\end{definition}
\medskip

\subsection{Semicontinuity for algebraically presented modules}\label{ssec.poly}
Let $A$ be Noetherian and $M$ finitely generated as $R=A[[x]]$-module.
Then $M$ is finitely $R$-presented and in this section we assume that $M$ has an
algebraic presentation matrix. That is, there exists a presentaion
$$ R^p \xrightarrow{\text {$T$}} R^q \to  M \to 0$$
with $T=(t_{ij})$ a $q\times p$ matrix such that $t_{ij} \in A\x$, $x=(x_1,\ldots,x_n)$, the ring of algebraic power series over $A$ (cf. Definition \ref{def.henselization}),  e.g. $t_{ij} \in A[x]$.  Under this assumption we shall prove the semicontinuity of $\hat d_\fb (M)$ for $\fb \in \Spec A$.\\

We set $R_0=A \langle x \rangle$ and $M_0=\coker (R_0^p  \xrightarrow{\text {$ T$}} R_0^p)$.  Then using the $\langle x \rangle$-adic completion we obtain $R_0^{\wedge}=R$ and $M_0^{\wedge}=M$.
\medskip

 \begin{lemma} \label{lem.10}
 Let $B\supset A$ be an $A$-algebra, $\fb \in \Spec B$ and $\fa = \fb \cap A$. Then
$$\hat d_\fa (M) < \infty \Leftrightarrow \hat d_\fb (M \hat \otimes_AB) < \infty \Leftrightarrow d^h_\fb (M_0 \otimes^h_AB) < \infty$$
and
$$\hat d_\fa (M) = \hat d_\fb (M\hat \otimes_AB)  =d^h_\fb (M_0 \otimes^h_AB). $$
  \end{lemma}
  
 \begin{proof} $M\hat \otimes_AB$ is considered as an $R\hat \otimes_AB =B[[x]]$-module and $M_0 \otimes^h_AB$ as $R_0 \otimes^h_AB = B\x$-module. Therefore we have 
 $$
\begin{array}{*{20}lcr}
	\hat d_\fa (M)&=&dim_{k(\fa)}(M \hat \otimes_Ak(\fa))\\
	\hat d_\fb (M\hat \otimes_AB)&=&dim_{k(\fb)}(M \hat \otimes_AB \hat \otimes_Bk(\fb))\\
	d^h_\fb (M_0\otimes^h_AB)&=&dim_{k(\fb)}(M_0 \otimes^h_AB \otimes^h_Bk(\fb))
\end{array}
$$
and
$$
\begin{array}{*{20}lcl}
M \hat \otimes_Ak(\fa) &=&  \coker (k(\fa)[[x]]^p   \xrightarrow{\text {$\bar T$}}  k(\fa)[[x]]^q)\\
M \hat \otimes_AB \hat \otimes_Bk(\fb) &=&  \coker (k(\fb)[[x]]^p   \xrightarrow{\text {$\bar T$}}  k(\fb)[[x]]^q)\\
M_0 \otimes^h_AB \otimes^h_Bk(\fb) &=&  \coker (k(\fb)\langle x \rangle^p   \xrightarrow{\text {$\bar T$}}  k(\fb)\langle x \rangle^q)\\
\end{array}
$$
 with $\bar T = (\bar t_{ij})$ and $\bar t_{ij}$ the induced elements in $k(\fa)[x]$ resp. $k(\fb)[x]$.\\
 
 If $\hat d_\fb (M\hat  \otimes_AB) < \infty $ there exists an $N_0$ such that  $\langle x \rangle^N M\hat \otimes_AB \hat\otimes_B k(\fb)=0$
for $N\geq N_0$  and hence
 $$ 
 \begin{array}{*{20}l}
 M \hat \otimes_AB \hat \otimes_Bk(\fb)
&=&  \coker (k(\fb)[[x]])^p/\langle x \rangle^N   \xrightarrow{\text {$\bar T$}}  (k(\fb)[[x]])^q /\langle x \rangle^N\\
 &=&  \big(\coker (k(\fa)[[x]]/\langle x \rangle^N)^p   \xrightarrow{\text {$\bar T$}}  (k(\fa)[[x]]/\langle x \rangle^N)^q\big) \otimes_{k(\fa)}k(\fb). 
  \end{array} 
  $$ 
 Since this holds for every $N\geq N_0$, we obtain  $\hat d_\fa (M) < \infty$. Similarly we can see that  $\hat d_\fa (M) < \infty$ implies $\hat d_\fb (M \otimes_AB) < \infty$
  and  in both cases we obtain $\hat d_\fa (M)=\hat d_\fb (M \otimes_AB) .$ This gives the first equality in the Lemma.
 Since $B\langle x \rangle/  \langle x \rangle^N = B[[x]]/  \langle x \rangle^N$   we get the remaining claims.
\end{proof}

\begin{lemma} \label{lem.hens}
Let $(A,\fm,\k)$ be a local Noetherian Henselian ring  and $R$ a local quasi-finite (i.e. $\dim_\k R/\fm R <\infty$) and finite type $A$-algebra in the Henselian sense\,\footnote{$R$ is an $A$-algebra of finite type in the Henselian sense if $R=A\langle t_1,\ldots,t_s \rangle$ for suitable $t_1,\ldots,t_s \in R$}.
Then $R$ is a finite $A$-algebra, i.e., finitely generated as an $A$-module. 
\end{lemma}

\begin{proof}
This is an immediate consequence of Proposition 1.5 of \cite{KPP78}.
\end{proof}

\begin{corollary} \label{cor.hens}
Let $(A,\fm,\k)$ be a local Noetherian Henselian ring and $R$ a local and finite type $A$-algebra  in the Henselian sense.
If $M$ is a finitely generated and quasi-finite (i.e. $\dim_{\k} M/\fm M < \infty$) $R$-module, then $M$ is a finitely generated $A$-module.
\end{corollary}
 
 \begin{proof}
Passing from $R$ to $R/\Ann_R (M)$ we may assume that $\Ann_R (M)=0$. In this case $\dim_{\k} M/\fm M < \infty$ implies $\dim_{\k} R/\fm R < \infty$.
Lemma \ref{lem.hens} implies that $R$ is a finitely generated $A$-module. Since $M$ is finitely generated over $R$ it follows that $M$ is a finitely generated $A$-module.
\end{proof}

\begin{theorem} \label{thm.hens}
Let $A$ be a Noetherian ring, $R = A[[x]]$, $x=(x_1,\ldots,x_n),$ and $M$ a finitely generated $R$-module admitting a presentation
$$ R^p \xrightarrow{\text {$T$}} R^q \to  M \to 0$$
with algebraic presentation matrix  $T=(t_{ij})$, $t_{ij} \in A[x]$ or, more generally, $\in A\x$. 
Fix $\fp \in \Spec A$ with $ \hat d_\fp(M) <\infty$. Then there is an open neighbourhood $U$ of $\fp$  in $\Spec A$ such that  
$$\hat d_\fq(M) \leq  \hat d_\fp(M) \ \text {for all } \fq \in U.$$
\end{theorem}

\begin{proof}
Recall that $R_0=A \langle x \rangle$ is the Henselization of $A[x]$ with respect to $\langle x \rangle$ and $M_0=\coker (R_0^p  \xrightarrow{\text {$ T$}} R_0^p)$.
Denote by $A^h$  the henselization of the local ring $A_\fp$ with respect to its maximal ideal. We set 
$R^h :=  A^h  \langle x \rangle$ and\,\footnote{Note that $R^h$ is the Henselization of $A_\fp[x]$ with respect to the maximal ideal 
$\langle\fp,x\rangle.$} 
$M^h :=  \coker \big((R^h)^p   \xrightarrow{\text {$T$}}  (R^h)^q\big)=M_0\otimes^hR^h.$
Then 
Lemma \ref{lem.10} implies $\hat d_\fp(M)=d^h_\fp(M^h)$
and Corollary \ref{cor.hens} that  $M^h$ is a finitely generated $A^h$-module ($R^h$ is a finite type $A^h$-algebra in the Henselian sense).
Lemma \ref{lem.hensel} implies that there is an \'etale neighbourhood $\pi: \Spec B \to \Spec A$ of $\fp$ such that 
$M_0 \otimes^h_AB= \coker \big((R_0 \otimes^h_AB)^p \xrightarrow{\text {$T$}} (R_0 \otimes^h_AB)^q\big)$
is a finitely generated $B$-module and $M_0 \otimes^h_AB\otimes^h_BA^h=M^h$.
Choose $\fb \in \Spec B$ such that $ \fb \cap A=\fp.$ This is possible because of Lemma \ref{lem.hensel}. Corollary \ref{cor.hens} and Lemma \ref{lem.Afin} imply that there is an open neigbourhood $\tilde U \subset \Spec B$ of $\fb$ 
such that for $\fc \in \tilde U$ we have $d_\fc(M_0 \otimes^h_AB) \le d_{ \fb}(M_0 \otimes^h_AB).$
Since $\pi $ is \'etale it is open (Proposition \ref{prop.etal}), $U:= \pi (\tilde U)$ is an open neighbourhood of $\fp$ in $\Spec A$ and for any $\fq \in U \cap \Spec A$ there exists a $\fc \in \tilde U \cap \Spec B$ with $\fc \cap A = \fq$.
From Lemma \ref{lem.10} we obtain 
$\hat d_\fq(M)=d_\fc(M_0\otimes^h_AB) \le d_{ \fb}(M_0 \otimes^h_AB)=\hat d_\fp(M).$
\end{proof}

The important property of Henselian local rings is that quasi-finite implies finite (in the sense of Corollary \ref{cor.hens}). Examples of Henselian local rings are quotient rings of the algebraic power series rings $A=\k\langle y\rangle/I$ over some field $\k$,  and analytic $\k$-algebras.\footnote{An analytic $\k$-algebra is the quotient  $\k\{y\}/I$, $y = ( y_1,\ldots,y_s)$, of a convergent power series ring over a complete real-valued field $\k$ (cf. \cite{GLS07}).
 E.g.,  if  $\k$ is any field with the trivial valuation, then $\k\{y\} =\k[[y]]$ is the formal power series ring; if  $\k \in \{\R, \C\}$,  then  $\k\{y\}$ is the usual convergent power series ring. 
}  

\medskip

If $A$ is a complete local ring containing a field, then any finitely generated $R$-module $M$ can be polynomially presented and semicontinuity of $\hat d_\fp(M)$ holds, as we show now.  We start with the following proposition, based on the Weierstrass division theorem.

\begin{proposition}\label{prop.complete}
Let $(A,\fm,\k)$ be  a Noetherian complete  local ring containing $\k$, $R=A[[x]]$, $x=(x_1,\dots,x_n)$,
and  $M$ a finitely generated $R$-module such that $\dim_{\k}M/\fm M <\infty$. Let $J=\Ann_R(M)$.
Then there exist $f_1,\ldots,f_s$, $s \geq n$, with the following properties:
\begin{enumerate}
\item $f_i \in A[x]$ for all $i$.
\item $f_{n-i+1} \in A[x_1,\ldots,x_i]$ is a Weierstrass polynomial with respect to $x_i$ for $i=1,\ldots,n$.
\item $J=\langle f_1,\ldots,f_s\rangle A[[x]]$.
\end{enumerate}
\end{proposition}
\begin{proof}
To prove the statements we use induction on $n$, the number of the variables $x$. 
The assumption implies that  $\dim_{\k}(R/J+\fm R) <\infty$, i.e. the ideal $J+\fm R$ is primary to the maximal ideal $\x + \fm R$ of $R$.
This implies that $x_n^b \in J+\fm R$ for some $b$. Therefore there exists $g \in J$, $g=x_n^b+f$ with $f \in \fm R$.  We know by Cohen's structure theorem that  $A = \k[[y]]/I$ for suitable variables $y$ and an ideal $I \subset \k[[y]]$. We can apply in the following the Weierstrass preparation and division theorem to representatives in $\k[[y,x]]$ and then take residue classes mod $I$. Obviously $g$ is $x_n$-general.
The Weierstrass preparation theorem implies  $g=uh$, u a unit in $R$, and $h \in A[[x_1,\ldots,x_{n-1}]][x_n]$ a Weierstrass polynomial with respect to $x_n$.
To simplify the notation we assume that $g$ is already a Weierstrass polynomial with respect to $x_n$. Setting $R_0=A[[x_1,\ldots,x_{n-1}]]$,
the Weierstrass division theorem (cf. \cite[Theorem I.1.8]{GLS07}) says that for 
any $f$ in $R$ there exist unique $h\in R$ and $r\in  R_0[x_n]$ such that
$f = hg+ r$, $\deg_{x_n}(r) \leq b-1$. In other words, as $R_0$ modules we have
\begin{align}
\tag{*}\label{*}
R = R\cdot g \oplus R_0 \cdot x_n^{b-1} \oplus  R_0 \cdot x_n^{b-2} \oplus \cdots \oplus R_0.
\end{align} 
We may thus assume that $J=\langle g_1,\ldots,g_r \rangle$ with $g_1=g$ and $g_i \in R_0[x_n]$ with $\deg_{x_n}(g_i) \leq b-1$.

If $n=1$ then $R_0=A$ and the claim follows from (\ref{*}).
If $n \geq 2$ then $M$ is a finitely generated $R_0$-module since
\begin{itemize}
\item $R/\langle g \rangle$ is finite over $R_0$ and
\item $g \in \Ann_R(M)$, i.e. $M$ is a finitely generated $R/\langle g \rangle$-module.
\end{itemize}
Now let $J_0=\Ann_{R_0}(M)$.
By induction hypothesis there are  $f_2,\ldots,f_l$, $l\geq n$,  such that 
\begin{enumerate}
\item $f_i \in A[x_1,\ldots,x_{n-1}]$ for all $i$.
\item $f_{n-i+1} \in A[x_1,\ldots,x_i]$ is a Weierstrass polynomial with respect to $x_i$ for $i=1,\ldots,n-1$.
\item $J_0= \langle f_2,\ldots,f_l\rangle R_0$.
\end{enumerate}
Now denote by $f_1$ be the remainder of the division of $g$ successively by $f_2,\ldots,f_n$ and by  $f_{l+i}$ the remainder of $g_i$ by $f_2,\ldots,f_n$ for $i>1$. These are polynomials in  $x_1, \dots, x_n$.
Then $f_1,\ldots,f_s$ satisfy the conditions 1. to 3. of the proposition.
\end{proof}
\begin{corollary} \label{cor.complete}
Let $(A,\fm,\k)$ be  a Noetherian complete  local ring containing $\k$, $R=A[[x]]$ and $M$ a finitely generated $R$-module such that $dim_{\k}M/\fm M <\infty$. Then 
$M$ is polynomially presented. 
\end{corollary}
\begin{proof} Assume $M$ has a presentation matirx $T =(g_{ij})$, $g_{ij} \in A[[x]]$.
Let $J=\Ann_R(M)$. The assumption implies that $\dim_{\k}R/(J+\fm R) <\infty$. Using Proposition \ref{prop.complete} we obtain  that $R/J$  is a $A$-finite and
$J=\langle f_1,\ldots,f_s \rangle $ with $f_{n-i+1} \in A[x_1,\ldots,x_i]$ a Weierstrass polynomial with respect to $x_i$ for $i=1,\ldots,n, n \leq s$.
This implies that $M$ has a presentation as $R/J$-module 
with presentation matrix $T$ having entries in $R/J$. Now we can divide representatives in $R$  of the entries of $T$ successively by the Weierstrass polynomials $f_{n-i+1}$, $i=1,\dots,n$. The remainders are polynomials in $A[x]$ representing the entries of $T$, which proves the claim.
\end{proof}
\medskip

    Let us collect the cases for which we proved that semicontinuity of $\hat d_\fp(M)$ holds. 
  \begin{corollary}  \label{cor.qf}
  Let $A$ be Noetherian and $M$ a finitely generated $R$-module.
Let $\fp \in \Spec A$ and assume that one of the following conditions is satisfied:
\begin{enumerate}
\item $M$ is finitely generated as $A$-module, e.g. $\Supp_R (M) \subset V(\x)$, or
\item $\dim A = 1$, or 
\item $M$ is algebraically $R$-presented, or
\item  $(A,\fm,\k)$ is a complete local ring containing a field\,\footnote{By Cohen's structure theorem this is equivalent to $A\cong \k[[y]]/I$, $\k$ a field.}.
\end{enumerate}
\noindent 
Then there is an open neighbourhood $U \subset \Spec A$  of $\fp$ such that $\hat d_\fq(M) \leq  \hat d_\fp(M) $ for all $\fq \in U.$
In particular, the quasi-completed-finite locus $\{\fp \in \Spec A \, | \,  \hat d_\fp (M) < \infty \}$ is open.
 \end{corollary}
 
\begin{proof}
Statement 1. follows from Proposition \ref{prop.ann}, statement 2. from Theorem \ref{thm.1dim} and 3. from Theorem \ref{thm.hens}. Statement 4. follows from 3. and Corollary \ref{cor.complete}.
\end{proof}
We do not know if semicontinuity of $\hat d_\fp(M)$ holds in general for $A$ Noetherian of any dimension and $M$ finitely but not necessarily algebraically presented over $R$.
\smallskip 

\begin{remark}\label{rm.conv} {\em
For completeness we recall cases where semicontinuity of the usual fibre dimension $d_\fp(M)$ on $\Spec A$  holds if $M$ is an arbitrary finitely presented $R$-module, for different (local) rings $A$ and $R$. 
\begin{itemize}
\item $(A,\fm,\k)$ local Noetherian Henselian, $R$ a finite type $A$-algebra in the Henselian sense
(by Corollary \ref{cor.hens} and Lemma \ref{lem.Afin}).
\item
$A= \k\{y\}/I$ an analytic $\k$-algebra and $R = \k\{y,x\}/J$ with $I\k\{y,x\} 
\subset J$ (by \cite[Theorem I.1.10]{GLS07}). 
\item  $A$ a Noetherian complete local ring containing a field, $R=A[[x]]$. 
This is a special case of the previous item. 
We mention it, since $R$ is of the form considered in this paper.
\item In the complex analytic situation with $A= \C\{y\}/I$, $y=(y_1,...,y_s)$, and $R = \C\{y,x\}/J$, $I\C\{y,x\} \subset J$, $x=(x_1,...,x_n)$, 
we have the following stronger statement: 
$A \to R$ induces a morphism of complex germs $f : (X, 0) \to (Y, 0)$, $(X,0)=V(J)\subset (\C^n \times \C^s,0), (Y,0)=V(I)\subset (\C^s,0)$ and $f$ the projection. For a sufficiently small representative $f:X \to Y$, $M$ induces a coherent $\ko_X$-module $\kf$ on $X$
and $d_{\fm}(M) <\infty$, $\fm$ the maximal ideal of $A=\ko_{Y,0}$,  means that the fibre dimension over $0\in Y$ is finite, i.e.
$ d_0(\kf):=\dim_\C \kf_{0}/\fm \kf_{0} < \infty.$ Then, for sufficiently small suitable $X$ and $Y$, is $f | \Supp \kf$ is a finite morphism and $ f_{\ast}\kf$ is a coherent $\ko_Y$-module (cf. \cite[Theorem I.1.67]{GLS07}). It follows that
$$ d_y(\kf) \,:= \dim_\C f_{\ast}\kf \otimes_{\ko_{Y,y} } \C =
\! \sum_{z\in f^{-1}(y)} \dim_\C\kf_z/\fm_y\kf_z$$
is upper semicontinuous at $0\in Y$,  i.e.
$0$ has an open neighbourhood $U \subset Y$  such that  
$ d_y(\kf) \leq   d_0(\kf) \text { for all }  y \in U.$
\end{itemize}
In the above cases $M$ is finite over $A$ if it is quasi-finite over $A$ and hence semicontinuity of $d_\fp(M)$ holds by Lemma \ref{lem.Afin}.
Example \ref{ex.Kt} shows that for $A$ an affine ring
and $R=A[[x]]$ 
semicontinuity of $d_\fp(M)$ does in general not even hold for polynomially presented modules. 

}
\end{remark} 
\medskip

\subsection{Related results}\label{ssec.related}

Instead of families of power series let us now consider families of algebras of finite type, a situation which is quite common in algebraic geometry. We treat the more general case of families of modules.

Let $A$ be a ring, $R=A[x]/I$ of finite type over $A$ and $M$ a finitely presented $R$-module.  $M$ is called {\em quasi-finite at $\fn \in \Spec R$} over $A$  if $\dim_{k(\fp)} M_\fn/\fp M_\fn < \infty$ with $\fp \in \Spec A$  lying under $\fn$. $M$ is called {\em quasi-finite over $\fp \in \Spec A$} if it is quasi-finite at all primes $\fn \in \Spec R$ lying over $\fp$, and $M$ is
{\em quasi-finite over $A$} if it is quasi-finite at all primes $\fn \in \Spec R$.
The following proposition is a generalization of results from \cite{Sta19}, where the case of ring maps is treated. 

\begin{proposition}\label{prop.ftype}
Let  $A$ be a ring, $R$ an $A$-algebra of finite type over $A$, $M$ a finitely presented $R$-module and $f: \Spec R \to  \Spec A$  the induced map of schemes. 
\begin{enumerate}
\item The following are equivalent:
\begin{enumerate}
\item $M$ is quasi-finite over $A$, 
\item $d_\fp(M) = \dim_{k(\fp)} M(\fp) = 
\sum_{\fn \in f^{-1}(\fp)}  \dim_{k(\fp)} M_\fn/\fp M_\fn < \infty$
$\forall \fp \in \Spec A$,
\item The induced map $A \to S:= R/ \Ann_R(M)$ is quasi-finite.
\end{enumerate}
\item (Zariski's main theorem for modules). The {\em quasi-finite locus of $M$}
$$\{\fn \in \Spec R \, | \, M \text{ is quasi-finite at } \fn \}$$ 
is open in $\Spec R$.
\end{enumerate}
\end{proposition}

\begin{proof} 1. $(a) \Rightarrow (b)$: We have to show that the support of $M(\fp)$ is finite.
By  \cite[Lemma 29.19.10]{Sta19}, if $R =A[x]/I$ is a ring  of finite type and quasi-finite over $A$, the induced map  $f: \Spec R \to  \Spec A $
has finite fibres $R(\fp) = R \otimes_A k(\fp) = k(\fp)[x]/I(\fp)$. It follows that 2. holds if  $M$ is a ring of finite type over $A$.

In the general case let $I = \Ann_R(M)$. Then $S=R/I$ is of finite type over $A$,  $M$ is finitely presented over $S$ and hence $\Supp(S) = \Supp(M)$. Moreover, let $J(\fp)$ be the annihilator of the finitely generated $R(\fp)$-module $M(\fp)$ satisfying $V(J(\fp)) = \Supp (M(\fp))$.
Since $R(\fp)$ is Noetherian and  $\dim_{k(\fp)} M_\fn(\fp) < \infty$ by assumption,
we have $\fn^N M_\fn(\fp)=0$ for some $N$ by Nakayama's lemma. Hence 
$\fn^N \subset J(\fp)R_\fn(\fp)$ and  $\dim_{k(\fp)} R_\fn‚(\fp)/J(\fp) R_\fn‚(\fp)<\infty.$ In general, the annihilator is not compatible with base change, hence $I(\fp)$ is in general different from $J(\fp)$. But for a finitely presented module the annihilator coincides up to radical with a Fitting ideal, and Fitting ideals are compatible with base change. It follows that  $\sqrt {J(\fp)R_\fn(\fp)} = \sqrt {I(\fp)R_\fn(\fp)}$ and  
therefore $\dim_{k(\fp)} R_\fn(\fp)/I(\fp)R_\fn(\fp) = \dim_{k(\fp)} S_\fn (\fp) <\infty$, which means that $A \to S$ is quasi-finite at $\fn$ (\cite[Definition 10.121.3]{Sta19}). Since this holds for each $\fn \in \Spec R$, the map $A \to S$ is quasi-finite and by  \cite[Lemma 29.19.10]{Sta19}
the set $\Supp (M(\fp)) = \Supp (S(\fp))$ is finite. 
 
$(b) \Rightarrow (c)$: If $d_\fp(M) < \infty$ for all $\fp \in \Spec A$, then 
$\dim_{k(\fp)} M_\fn/\fp M_\fn <\infty$ for all $\fn$ and $\fp$ under $\fn$. 
Then $A \to S$ is quasi-finite by the previous step.

$(c) \Rightarrow (a)$: If $A \to S$ is quasi-finite then $\dim_{k(\fp)} S_\fn (\fp) <\infty$
for all $\fp$ and $\fn$ over $\fp$. Since $M_\fn (\fp)$ is finitely presented as 
$S_\fn (\fp)$-module, $\dim_{k(\fp)} M_\fn (\fp) <\infty$ for all $\fp$ and $\fn$ over $\fp$ and $M$ is quasi-finite over $A$.

2. In the first and third step of 1. we proved 
$\dim_{k(\fp)} S_\fn (\fp) <\infty$ if and only if $\dim_{k(\fp)} M_\fn (\fp) <\infty$, and
hence  $M$  is quasi-finite at  $\fn$ iff  $A\to S$  is quasi-finite at  $\fn$. 
It follows from a version of Zarisk's main theorem as proved in 
 \cite[Lemma 10.122.13]{Sta19} that the set $\{\fn \in \Spec S\, | \,  A\to S \text{ is quasi-finite at } \fn \}$ is open in  $\Spec S$ and thus of the form $U \cap S$ with $U$ open in   $\Spec R$. 
 If  $\fn \in V= \Spec R \setminus \Spec S$ then $M_\fn(\fp) =0$, hence $M$ is quasi-finite at $\fn \in V$. Thus, the quasi-finite locus of $M$ is the open set $U \cup V$.
\end{proof}

\begin{example} {\em
In the situation of Proposition \ref{prop.ftype}, although the quasi-finite locus of $M$ is open in $\Spec R$, we cannot expect semicontinuity of $d_\fp(M)$ on $\Spec A$. We give an example showing that the vanishing locus of $d_\fp(M)$ is not open in $\Spec A$:
Let $K$ be an algebraically closed field, $A=K[y]$, $R = A[x]$ and $M = R/ \langle xy-1\rangle$. Then $M$ is quasi-finite over $A$ but $d_\fp(M)$ is not semicontinuous since  $d_\fp(M) = 0$ if $\fp=\langle y\rangle $ and $d_\fp(M) = 1$ otherwise. 
}
\end{example}

By Corollary  \ref{cor.qf}.4, semicontinuity of $\hat d_\fp(M)$ holds for $M$ a finitely generated $A[[x]]$-module if $(A, \fm)$ is a complete Noetherian local ring  containing a field under the assumption that $\Supp_A(M)=A$ but $\Supp_R(M ) \not\subset V(\x)$ (the difficult case). The question arose whether completeness was necessary. The following example shows that this is not the case.

\begin{example}\label{ex.referee}{\em
We give an example of a non-complete local ring $(A,\fm)$ and  a finitely presented $R=A[[x]]$-module $M$ which is also a finitely presented $A$-module with $\Supp_R(M)$ being not contained in $V(\x)$ and $\Ann_A (M) =0$.\\
Let $\k$ be a field and $t_1,t_2$ independent variables. Let $A=\k[t_1]_{\langle t_1\rangle}[[t_2]]$ and $R=A[[x]]$ with x a single variable.
The ring $A$ is local with maximal ideal $\fm=\langle t_1,t_2 \rangle A$ and not complete. Let $M=R/\langle x-t_2 \rangle$. Then $\Ann_R(M)=\langle x-t_2 \rangle$
and $\Supp_R(M)=V(\langle x-t_2 \rangle \not\subset V(\langle x \rangle )$ and $\Ann_A(M)=\langle x-t_2 \rangle \cap A=(0)$. Since $M$ is polynomially presented, 
$\hat d_\fp(M)$ is semicontinuous.}
\end{example}

Let $M$ be a finitely presented $R=A[[x]]$-module and also finitely presented as an $A$-module with $A$ Noetherian.  In Proposition \ref{prop.ann} we have shown the semicontinuity of $d_\fp(M)$ and $\hat d_\fp(M)$ on $\Spec A$, as well as the inequality 
$\hat d_\fp(M) \leq d_\fp(M).$ The following example shows that $\hat d_\fp(M) < d_\fp(M)$ may happen.

\begin{example} \label{ex.complete} {\em
A modification of Example \ref{ex.Kt} shows that $\hat d_{\fp}(M) \neq d_{\fp}(M)$ may happen for $A$ a power series ring.
Let $A=K[[t]]$, $K$ a field, $R=A[[x]]$, and let
$M = R/\langle t-x\rangle \cong K[[t]]$.  For the two prime ideals $\langle 0\rangle$
and $\langle t\rangle$ of $A$ we get:

$k(\langle 0\rangle) = K((t))$,  $k(\langle t\rangle) = K$,
$M(\langle 0\rangle) \cong K[[t]]\otimes_{K[[t]]}K((t)) = K((t))$, $M(\langle t\rangle) \cong K$
and $d_{\langle 0\rangle}(M) =d_{\langle t\rangle}(M) =1$. Hence  $d_{\fp}(M)$
is semicontinuous (even continuous) on $\Spec A$  as predicted in Remark \ref{rm.conv}, third item.

$\hat M(\langle 0\rangle) \cong K((t))[[x]]/\langle t-x\rangle =0$, $\hat M(\langle t\rangle) \cong K$
and $\hat d_{\langle 0\rangle}(M) =0$, $\hat d_{\langle t\rangle}(M) =1$. 
Hence  $d_{\fp}(M)$
is semicontinuous on $\Spec A$ as predicted by  Corollary \ref{cor.qf}.
Note that $M$ is finitely presented as $A$-module and we have 
$\hat d_{\langle 0\rangle}(M) < d_{\langle 0\rangle}(M)$.
}
\end{example} 

\section{Singularity invariants}\label{sec.singinv}
\subsection{Isolated singularities and flatness}

Recall that a local Noetherian ring $(A,\fm)$ is said to be {\em regular} if $\fm$  can be generated by $\dim A$ elements. A Noetherian ring $A$ is said to be regular if the local ring $A_\fp$ is regular for all $\fp \in \Spec A$.
For arbitrary Noetherian rings the {\em regular locus}
$\Reg A := \{\fp \in \Spec A \, | \, A_\fp \text{ is regular} \}$
need not be open in $\Spec A$. However, 
$\Reg A $ is open if $A$ is a complete Noetherian local ring 
(\cite[Corollary of Theorem 30.10]{Mat86}) and the {\em non-regular locus} 
$\{\fp \in \Spec A \, | \, A_\fp \text{ is not regular} \}$
is closed. 

However, in our situation of families of power series, the notion of formal smoothness is more appropriate than that of regularity. Formal smoothness is a relative notion and refers to a morphism, while regularity is an absolute property of the ring. The notions are related as follows. Let $(A,\fm)$ be a local ring containing a field $\k$. If $A$ is formally smooth over $\k$ (w.r.t. the $\fm$-adic topology) then $A$ is regular and the converse holds if the residue field  $A/\fm$ is separable over $\k$ (see Remark \ref{rm.smooth}). Hence formal smoothness of $A$ over $\k$ coincides with regularity if $\k$ is a perfect field. The notions  also coincide for arbitrary 
 $\k$ if $A$ is the quotient ring of a formal power series ring over  $\k$ by an ideal (cf. Lemma \ref{lem.jac}).
 \medskip

We recall now basic facts about formal smoothness. For details and proofs see \cite{Mat86} and \cite{Maj10}.
\begin{definition}
Let $A$ be a ring, $B$ an $A$-algebra defined by $\phi:A \longrightarrow B$ and $I$ an ideal in $B$. The $A$-algebra $B$ is called {\em formally smooth with respect to the
$I$-adic topology} (for short {\em $B$ is $I$-smooth over $A$}) if for any $A$-algebra $C$ and any continuous\footnote{ Here we consider $B$ with the $I$-adic topology and $C/N$ with the discrete topology; $u$ is continuous if $u(I^m)=0$ for some $m$.}  $A$-algebra homomorphism $u:B \longrightarrow C/N$,
$N$ an ideal in $C$ with $N^2=0$, there exist $\sigma:B\longrightarrow C$ such that $\pi\sigma=u$.
\[
\xymatrix{B \ar[r]^-{u}\ar@{-->}[dr]^\sigma & C/N\\
A \ar[u]^\phi\ar[r]_v & C\ar[u]_\pi
}
\]
If $I=0$ then $B$ is called a {\em formally smooth} $A$-algebra.
\end{definition}

\begin{remark}\label{rm.smooth}{\em
\begin{enumerate}
\item A formally smooth map of finite presentation is smooth (\cite{Sta19} Proposition 10.137.13).
\item $A[x]$, $x = (x_1,\dots, x_n)$, is smooth over $A$ (\cite{Sta19} Lemma 10.137.4).
\item $A[[x]]$ is $\x$-smooth over $A$ (\cite{Mat86} page 215).
\item Let $(A,\fm)$ be a local ring containing a field $\k$. 
\begin{enumerate}
\item $A$ is $\fm$-smooth over $\k$ iff $A$ is geometrically regular, i.e. $A\otimes_\k \k'$ is a regular ring for every finite extension field $\k'$ of $\k$
(\cite[Theorem 28.7 ]{Mat86}).
\item Assume that $A/\fm$ is separable over $\k$. Then $A$ is $\fm$-smooth over $\k$ iff $A$ is regular (\cite{Mat86} Lemma 1, page 216).
\end{enumerate}
\end{enumerate}}
\end{remark}

We now generalize example 1 on page 215 of \cite{Mat86}.
\begin{lemma}\label{lem.ismooth}
Let $A$ be a ring, $B$ a $A$-algebra, $I$ an ideal in $B$ and $\widehat B$ the $I$-adic completion of $B$.
$\phi:A\longrightarrow B$ is $I$-smooth iff $\hat\phi:A \longrightarrow \widehat B$ is $I\widehat B$-smooth.
\end{lemma}

\begin{proof}
Assume that $B$ is $I$-smooth over $A$ and consider the following commutative diagram:
\[
\xymatrix{\widehat B \ar[r]^-{\hat u}\ar@{-->}[dr]^{ \hat\sigma} & C/N\\
A \ar[u]^{\hat \phi}\ar[r]_v & C\ar[u]_\pi
}
\]
with $N^2=0$. We have to prove that there exists $\hat \sigma$ such that $\pi\hat \sigma=\hat u$.
Since $\hat u$ is continuous there exist $m$ such that $\hat u(I^n\widehat B)=0$. Let $i:B\longrightarrow \widehat B$ be the cononical map
such that $\hat \phi=i \phi$. The $I$-smoothness of $B$ implies that there exists $\sigma:B\longrightarrow C$ such that
$\sigma\pi=\hat ui$. $\hat u(I^n\widehat B)=0$ implies $\sigma(I^m) \subset N$. Since $N^2=0$ we obtain $\sigma(I^{2m})=0$.
We obtain the following commutative diagram:
\[
\xymatrix{& \widehat{B}\ar[d]^{\hat i_{2m}} & & & \\
B\ar[ur]^i\ar[r]^{i_{2m}} & B/I^{2m}\ar[r]^{u_{2m}}\ar@{-->}[dr]^{\sigma_{2m}} & C/N\\
& A\ar[r]_v\ar[ul]^\phi & C\ar[u]_\pi
}
\]
Now we define $\hat \sigma=\sigma_{2m}\hat i_{2m}$. This proves that $\hat\phi:A \longrightarrow \widehat B$ is $I\widehat B$-smooth.

Now assume that $\hat\phi:A \longrightarrow \widehat B$ is $I\widehat B$-smooth. Consider the following commutative diagram:
\[
\xymatrix{B \ar[r]^-{u}\ar@{-->}[dr]^\sigma & C/N\\
A \ar[u]^\phi\ar[r]_v & C\ar[u]_\pi
}
\]
with $N^2=0$. We have to prove that there exists $\sigma$ such that $\pi \sigma=u$.
Since $\hat\phi:A \longrightarrow \widehat B$ is $I\widehat B$-smooth there exists $\hat \sigma:\widehat B\longrightarrow C$ with
$\pi\hat \sigma=\hat u$. Now we define $\sigma=\hat \sigma i$ and obtain $\pi \sigma=u$.
\end{proof}

%

The following  important theorem is due to Grothendieck (\cite{Mat86} Theorem 28.9).
\begin{theorem}\label{thm.Groth}
Let $(A,\fm)$ be a local ring and $(B,\fn)$ a local $A$-algebra. Let $\bar B=B/\fm B$ and $\bar \fn=\fn/\fm B$. Then $B$ is $\fn$-smooth over $A$ iff $\bar B$ is $\bar \fn$-smooth over $A/\fm$ and $B$ is flat over $A$.
\end{theorem}

\begin{definition} 
Let $A$ be a ring and $B$ an $A$-algebra defined by $\phi:A \longrightarrow B$.
We define the {\em smooth locus}  of $\phi$ by
\begin{center}
$\Sm(\phi): =\{P \in \Spec(B) | A_{\phi^{-1}(P)} \longrightarrow B_P$ is $P$-smooth$\}$.
\end{center}
and the {\em singular locus} of $\phi$ by
$$\Sing(\phi) := \Spec(B) \smallsetminus \Sm(\phi).$$
\end{definition}

\begin{remark} \label{rm.sm}{\em
Let $A$ be a ring and $B$ an $A$-algebra defined by $\phi:A \longrightarrow B$. The Theorem \ref{thm.Groth} of Grothendieck implies that
$$
\begin{array}{*{20}l}
\Sm(\phi)=&  \! \! \{P \in \Spec(B) \, | \,A_{Q} \longrightarrow B_P, Q=\phi^{-1}(P), \text{is flat}\\
 	 &\text{and }  B_P/QB_P \text{ is } PB_P/QB_P-\text{smooth over } k(Q) \}.
\end{array}
$$
}
\end{remark}
\medskip

Now let $\k$ be a field, $\k[[x]], \,  x = (x_1,\cdots, x_n),$ the formal power series ring over  $\k$ and  $I$ an ideal in  $\x \k[[x]]$. If $I$ is generated by  $f_1,\ldots,f_m$ we denote by $Jac(I)$ the Jacobian matrix $(\partial f_j / \partial x_i)$ and by  
$I_k(Jac (I))$ the ideal generated by the $k\times k$-minors of $Jac(I)$ (which is independent of the chosen generators $f_j$). The following lemma gives  equivalent conditions for the maximal ideal $\x \in B = \k[[x]]/I$ to be contained in the smooth locus $\Sm(\phi)$ of the map $\phi: \k \to B$ (Remark \ref{rm.sm}).

\begin{lemma} \label{lem.jac}
If $\dim \k[[x]]/I =d$ the following are equivalent.
\begin{enumerate}
\item $\k[[x]]/I$ is $\x$-smooth over $\k$.
\item $\k[[x]]/I$ is regular. 
\item $I_d(Jac (I)) = \k[[x]]$ {\em (Jacobian criterion)}.
\item $\k[[x]]/I \cong \k[[y_1,\dots,y_d]].$ 
\end{enumerate}
\end{lemma}

\begin{proof}
The equivalence of 1. and 2. follows from \cite[Lemma 1, p. 216]{Mat86},
the equivalence of 3. and 4. is the inverse mapping theorem for formal power series\,\footnote{Given $f_1,\ldots,f_n \in \k[[x_1,\ldots,x_n]]$ then $det(\frac{\partial{f_i}}{\partial{x_j}})$ is a unit iff $\k[[x_1,\ldots,x_n]]=\k[[f_1,\ldots,f_n]]$ (\cite[Theorem I.1.18]{GLS07}).}.
Obviously 4. implies 2. From \cite[Theorem 29.7, p. 228 in]{Mat86} we deduce that 2. implies 4.
\end {proof}

\begin{remark}{\em 
Part of the lemma can be generalized by extending the proof of Theorem 30.3 in \cite{Mat86} as follows:\\
Let $P$ be a prime ideal in $\k[[x]]$ containing $I=\langle f_1,\ldots,f_m\rangle$ and $\fm$ the maximal ideal of $A=\k[[x]]_P/I\k[[x]]_P$. Then
$I_d(Jac (I)) = \k[[x]]_P$ implies that $A$ is $\fm$-smooth over $\k$ (or geometric regular by Remark \ref{rm.smooth}).}
\end{remark}

We use the Jacobian criterion to define the singular locus of ideals in power series rings over a field.

\begin{definition} \label{def.sing}
\begin{enumerate}
\item If  $B=\k[[x]]/I$ is pure $d$-dimensional (i.e. $\dim B/P =d$ for all minimal primes $P \in \Spec B$) we define the {\em singular locus of $B$ (or of $I$)} as 
$$\Sing (B) = V(I+ I_d(Jac(I)).$$
\item If $B$ is not pure dimensional we consider the minimal primes $P_1,\dots,P_r$ of $B$. Then $B/P_i$ is pure dimensional and we define the singular locus of $B$ as
\begin{align*}
\Sing (B) = \bigcup\nolimits_{i=1}^r \Sing (B/P_i) \cup  \bigcup\nolimits_{i\neq j}V(P_i)\cap V(P_j),
\end{align*}
which is a closed subscheme of $\Spec B$. The points in $\Spec B \setminus \Sing (B)$ are called {\em non-singular} points of $B$.
\item We say that  $\k[[x]]/I$ (or $I$) has an {\em isolated singularity} (at 0) if the maximal ideal  $\x$ is an isolated point of $\Sing(\k[[x]]/I)$ or if $\x$ is a non-singular point.
\end{enumerate}
\end{definition}

\begin{remark}
Let $i:\k \longrightarrow B$ be the obvious inclusion. Then $\Sing(B)=\Sing(i)$ whenever $\k$ is perfect, but not in general.
A counterexample is given for instance, by letting $\chara( \k)=p>0$,  $a \in \k \setminus \k^p$ and $B=\k[[x_1,x_2]]/\langle x_1^p-ax_2^p\rangle$.
\end{remark}
\noindent Note that $\Sing(B)$ carries a natural scheme structure given by the Fitting ideal 
$I+I_d(Jac(I)) \subset \k[[x]]$ if $B$ is pure $d$-dimensional. In general we endow 
 $\Sing(B)$ with its reduced structure.
\medskip

Now let us consider families. Let $A$ be a Noetherian ring, $F_1, \ldots, F_m \in \x A[[x]]$,  $I \subset  A[[x]]$ the ideal generated by  $F_1,\ldots, F_m $  and set $B:= A[[x]]/I$. We describe now the smooth locus of the map $\phi: A \to B$ along the section
$\sigma: \Spec A \to \Spec B, \, \fp \mapsto \fn_\fp = \langle x, \fp \rangle$, of $\Spec \phi$.

For $\fp \in \Spec A$ denote by $F_i(\fp)$ the image of $F_i$ in $k(\fp)[[x]]$. Note that $F_1(\fp),\ldots,F_m(\fp)$  generate the ideal $\hat I(\fp)  \subset k(\fp)[[x]]$, 
and that  we have (by Lemma \ref{lem.loc}.3) for the {\em completed fibre} of $\phi$ over $\fp$

$$\hat B(\fp)= (B_{ \fn_\fp}/\fp B_{ \fn_\fp})^\wedge = k(\fp)[[x]]/\hat I(\fp).$$
The maximal ideals of the local rings of the fibre $B_{ \fn_\fp}/\fp B_{ \fn_\fp}$ and the completed fibre $\hat B(\fp)$ are generated by $\fn_\fp/\fp =\x$. 
Assume that $\phi: A \to B$ is flat. Then the theorem of Grothendieck says 
$$ 
\begin{array}{*{20}l}
{\fn_\fp} \in \Sm(\phi)  &\Leftrightarrow& 
B_{\fn_\fp} \text { is } \fn_\fp \text {-smooth over } A_\fp \\
& \Leftrightarrow& B_{ \fn_\fp}/\fp B_{ \fn_\fp} \text { is } \x \text {-smooth over } k(\fp).
\end{array}
$$

\begin{lemma}  \label{lem.sing}
With the above notations assume that $\phi: A \to B$ is flat. Denote by 
$$\Sing_\sigma (\phi) := \{ \fn_\fp \in \Spec B \, | \,  B_{\fn_\fp} \text { is not } \fn_\fp \text {-smooth over } A_\fp \}$$
the {\em singular locus of $\phi$ along the section}
$\sigma$. Then
$$\Sing_\sigma (\phi) = \{ \fn_\fp \in \Spec B \, | \ \hat B(\fp)  \text { is not regular} \}.$$
\end{lemma}
\begin{proof}
$B_{ \fn_\fp}/\fp B_{ \fn_\fp}$  is $\x$-smooth over  $k(\fp)$ iff $(B_{ \fn_\fp}/\fp B_{ \fn_\fp})^\wedge =  k(\fp)[[x]]/\hat I(\fp)$  is $\x$-smooth over  $k(\fp)$ by Lemma \ref{lem.ismooth}. The claim follows from Lemma \ref{lem.jac}.
\end{proof}

Since we assumed $B$ to be flat over $A$, we have 
$\dim \hat B(\fp) = \dim B_{\fn_\fp} - \dim A_\fp$ (by \cite[Theorem 15.1]{Mat86}).
If $\phi$ is of pure relative dimension $d$ (i.e. $\hat B(\fp)$ is pure $d$-dimensional for all $\fp$) then Lemma \ref{lem.jac} implies
$$\Sing_\sigma (\phi) = \{ \fn_\fp \in \Spec B \, | \, \hat I_d(Jac (I))(\fp) \text { is a proper ideal of } k(\fp)[[x]] \},$$
where $Jac(I)$ is the Jacobian matrix $(\partial F_j / \partial x_i)$ and
$I_d (Jac(I)) \subset A[[x]]$ the ideal defined by the $d\times d$-minors.
\medskip

\subsection{Milnor number and Tjurina number of hypersurface singularities} \label{ssec.mutau}
Let $\k$ be a field and $f \in \k[[x]],  x = (x_1,\cdots, x_n)$ a formal power series. The most important invaraints are the {\em Milnor number} $\mu(f)$ and the {\em Tjurina number}  $\tau(f)$, defined as
$$
\begin{array}{*{20}l}
\mu(f) & = & \dim_\k  \k[[x]]/j(f),\\
\tau(f) & = & \dim_\k  \k[[x]]/\langle f, j(f)\rangle,
\end{array}
$$
where $j(f) = \langle {\partial f}/{\partial x_1}, \ldots, {\partial f}/{\partial x_n}\rangle$ is the {\em Jacobian ideal} of $f$. We say that $f$ has an {\em isolated critical point} (at 0) resp. an {\em isolated singularity} (at 0) if $\mu(f) < \infty$
resp.  $\tau(f) < \infty$. Note that $\tau(f) < \infty$ iff  $\k[[x]]/\langle f \rangle$ has an isolated singularity in the sense of Definition \ref{def.sing}.

\begin{remark}\label{rm.char0}{\em
Let $\chara(\k) = 0$. It is proved in \cite[Theorem 2]{BGM12} that for $f\in\langle x\rangle$, 
 $\mu(f) < \infty \Leftrightarrow \tau(f) < \infty$  but it is easy to see that this is not true in positive characteristic. We have always  $\tau(f) \leq \mu(f)$  and
  $\tau(f) = \mu(f)  \Leftrightarrow f \in j(f)$.
 If $\k = \C$ and if $f \in \langle x\rangle^2$ has an isolated singularity, this is equivalent to $f$ being quasi homogeneous
  by a theorem of K. Saito (see \cite{Sa71}). His  proof generalises to any algebraically closed field of characteristic zero (cf. \cite[Theorem 2.1]{BGM11}).}
 \end{remark}
\medskip

We consider now families of singularities. Let $A$ be a Noetherian ring and $F\in R=A[[x]]$. Set 
$$j(F):=\langle {\partial F}/{\partial x_1}, \ldots, {\partial F}/{\partial x_n}\rangle$$ and for $\fp \in \Spec A$ denote by 
$F(\fp)$ the image of $F$ in $ k(\fp)[[x]]$.  Then the 
Milnor number 
$$ \mu (F(\fp)) =  \dim_{k(\fp)}  k(\fp)[[x]]/j(F(\fp))$$ 
and the Tjurina number 
$$\tau (F(\fp)) =  \dim_{k(\fp)}  k(\fp)[[x]]/\langle F(\fp),j(F(\fp))\rangle $$ 
are defined, and we deduce now the semicontinuity of $ \mu (F(\fp))$ and $\tau (F(\fp)).$

\begin{proposition} \label{prop.mutau}
Let $A$ be Noetherian, $F\in R=A[[x]]$ and  $\fp \in \Spec A$. Assume that 
$V(j(F))\subset V(\x)$ resp.  $V(\langle F,  j(F)\rangle)\subset V(\x)$ (as sets), or
 $\dim A =1$, or $F\in A[x]$, or $A$ is a complete local ring containing a field.
Then $\mu(F(\fp))$ and $\tau(F(\fp))$ are semicontinuous at $\fp \in \Spec A$. 
\end{proposition}

\begin{proof} 
Set 
$M=R/ j(F)$ resp. $M=R/ \langle F, j(F) \rangle$, then $\Supp_R (M) = V(j(F))$ resp.  $\Supp_R (M) = V(\langle F,  j(F)\rangle)$.
 Using Lemma \ref{lem.loc} we get $\hat d_\fq (M) = \mu(F(\fq))$ resp. $\hat d_\fq(M) = \tau(F(\fq))$ for $\fq \in \Spec A$. The result follows from 
 Corollary \ref{cor.qf}.
\end{proof}

\begin{corollary}  \label{cor.mutau}
Let $F \in \Z[x]$, $p\in \Z$ a prime number and denote by $F_p$ the image of $F$ in $\F_p[[x]]$ and by $F_0$ the image of $F$ in $\Q[[x]]$.

If $\mu(F_p)$ is  finite, then  $\mu(F_p) \geq \mu(F_0)$ and  $\mu(F_p)  \geq  \mu(F_q)$ for all except finitely many prime numbers $q\in \Z$. In particular, if 
$\mu(F_p)$ is  finite for some $p$ then $\mu(F_0)$ is finite.

If $\mu(F_0)$ is  finite, then $\mu(F_0) \geq \mu(F_q)$
(and hence ``='') for all except finitely many prime numbers $q\in \Z$.

The same holds for the Tjurina number.

\begin {example} {\em We illustrate the corollary by a simple example.
Let $F=F_0=x^p+x^{(p+1)}+y^q$ with $p,q$ prime numbers. Then $\mu(F_0)= (p-1)(q-1)$, 
$\mu (F_p) =p(q-1)$ $\ge \mu(F_0)$ while $\mu (F_q) =\infty$. Moreover, for any prime number $r \neq p,q$ we have $ \mu(F_r) = \mu(F_0)$.}
\end{example}
\end{corollary}

\medskip

\subsection{Determinacy of ideals} \label{ssec.det}
Let $I$ be a proper ideal of $\k[[x]]$ and
 $f_1, \ldots, f_m$ a minimal set of generators of $I$.
$I$ is called {\em contact $k$-determined} if for every ideal $J$ of $\k[[x]]$ that can be generated by $m$ elements $g_1, \ldots, g_m$ 
with $g_i-f_i\in\langle x \rangle ^{k+1}$ for $i=1,\ldots,m$, 
the local $\k$-algebras  $\k[[x]]/I$ and  $\k[[x]]/J$ are isomorphic.
$I$ is called {\em {finitely contact determined}} if $I$ is contact $k$-determined for some $k$. It is easy to see (cf. \cite[Proposition 4.3]{GPh19}) that these notions depend only on the ideal and not on the set of generators.

The ideal $I$ or the ring $\k[[x]]/I$ is called a {\em complete intersection} if $\dim \k[[x]]/I = n-m$ and an {\em  isolated complete intersection singularity (ICIS)} if it has moreover an isolated singularity.
\medskip

Set $f=(f_1,...,f_m) \in \k [[x]]^m$ and denote by  $\langle {\partial f}/{\partial x_1}, \ldots, {\partial f}/{\partial x_n}\rangle$ the submodule of $\k[[x]]^m$, generated by the m-tuples ${\partial f}/{\partial x_i} = ({\partial f_1}/{\partial x_i},\dots, {\partial f_m}/{\partial x_i})$,  $i=1,\ldots, n$.
We define
$$T_I :=  \k[[x]]^m\Big/I\k[[x]]^m+\langle {\partial f}/{\partial x_1}, \ldots, {\partial f}/{\partial x_n}\rangle.$$

If $I$ is a complete intersection, then  
$\tau(I) := \dim_\k T_I$
is called the {\em Tjurina number of I}. For a complete intersection $T_I$ is concentrated on the singular locus of $\k[[x]]/I$ (Definiton \ref{def.sing})
and $\tau(I)$ is finite iff $I$ has an isolated singularity. This follows from \cite[Lemma 3.1]{GPh19}, where it is shown that the ideals $I+I_{n-m}(Jac(I))$
and $\Ann_{\k[[x]]} (T_I)$ have the same radical.
\medskip

The module $T_I $ is used in the following theorem.

\begin {theorem}  [\cite{GPh19}, Theorem 4.6] \label{thm.GPh19}
Let $I \subset  \k[[x]]$ be a proper ideal  and $\k$ infinite. Then the following are equivalent:
\smallskip

(i) $I$ is finitely contact determined.  	
\smallskip

(ii) $\dim_\k T_I <\infty$.	
\smallskip

(iii)  $R/I$ is an isolated complete intersection singularity. 	
\smallskip

\noindent If one of these condition is satisfied then $I$ is contact $(2\dim_\k T_I -ord(I) +2)$-determined, where $ord(I) = \max \{k \ | \ I \subset \langle x \rangle^k \}$. The implications
$(iii) \Leftrightarrow (ii) \Rightarrow (i)$ hold for any field $\k$, as well as $(i) \Rightarrow  (ii)$ for hypersurfaces.
\end{theorem}

\begin{proposition}\label{prop.TI} 
Let $A$ be Noetherian, $F_1, \ldots, F_m \in \x A[[x]]$. Let  $I \subset  A[[x]]$ be the ideal generated by  $F_1,\ldots, F_m $  and $\hat I(\fp) \subset k(\fp)[[x]]$, $\fp \in \Spec A$,  the ideal generated by $F_1(\fp),\ldots,F_m(\fp) \in k(\fp)[[x]]$.

Assume that 
$V(I+I_{n-m}(Jac(I)))\subset V(\x)$ (as sets), or
 $\dim A =1$, or $F_i \in  \x A[x]$ for $i=1, \ldots, m$, or $A$ is a complete local ring containing a field.

Then
any $\fp \in \Spec A$ has an open neighbourhood $U \subset \Spec A$ such that for all $\fq \in U$  $\dim_{k(\fp)} T_{\hat I(\fp)} \geq  \dim_{k(\fq)} T_{\hat I(\fq)} $. 
\end{proposition}
\begin{proof} By \cite[Lemma 3.1]{GPh19} $\Supp (T_I) = V(I+I_{n-m}(Jac(I)))$. The claim follows from Corollary \ref{cor.qf}.
\end{proof}
\bigskip

\subsection{Tjurina number of complete intersection singularities} \label{ssec.icis}

We show first that being a regular sequence in a flat family of power series in $R=A[[x]]$ is an open property.

\begin{proposition}\label{prop.reg}
Let $A$ be a Noetherian ring, $F_i \in \x R$, $i=1, \ldots, m$ and $M$ a finitely generated $R$-module. For
$\fp \in \Spec A$ we denote by $F_i(\fp)$  the image of $F_i$ in $\hat R(\fp)= k(\fp)[[x]]$ and by $F_{i\fn_\fp}$  the image of $F_i$ in $R_{\fn_\fp}(\fp)$ (cf. Definition \ref{def.cf}). 

\begin{enumerate}
\item [(i)] If  $\fp \in \Spec A$  then $F_1(\fp),\ldots,F_m(\fp)$  is an $\hat M(\fp)$-sequence iff $F_{1\fn_\fp},\ldots,F_{m\fn_\fp}$ is an $M_{\fn_\fp}(\fp)$-sequence.
\item [(ii)] Let  $F_1,\ldots, F_m $ be an $M$-sequence and let $M/\langle F_{1},\ldots,F_{m}\rangle M$ be $A$-flat. Then
$F_1(\fp),\ldots,F_m(\fp)$  is  an $\hat M(\fp)$-sequence for all $\fp \in \Spec A$.
 \smallskip
 
 \item [(iii)] Let  $\fp \in \Spec A$ and
 $F_1(\fp),\ldots,F_m(\fp)$  an $\hat M(\fp)$-sequence. 
 If $M/\langle F_{1},\ldots,F_{m}\rangle M$ is flat over $A$, then there exists an  open neighbourhood $U$ of  $\fp$  in $\Spec A$ such that 
 $F_1(\fq),\ldots,F_m(\fq)$ is a $\hat M(\fq)$-sequence for all $\fq$ in $U$.
\end{enumerate}
\end{proposition}

\begin{proof}
Set $M_0=M,  M_i=M/\langle F_{1},\ldots,F_{i}\rangle M$ and  consider for 
$i=1,\ldots, m$ the exact sequence 
	\begin{align*}
	0 \to K_{i-1} \to M_{i-1}  \xrightarrow[\text {}]{F_i} M_{i-1} \to M_i \to 0,  \tag{*}\label{*}
	\end{align*}
with $K_{i-1}$ the kernel of $F_i$.

(i) By Lemma \ref{lem.loc}.2
 $\hat R(\fp) = R_{\fn_\fp}(\fp)^\wedge $ and  $\hat M_{i} (\fp) = 
 M_{i,{\fn_\fp}}(\fp)^\wedge$ for all $i$ and hence
 $$F_i(\fp) = F_{i\fn_\fp}^\wedge:  (M_{i-1,{\fn_\fp}}(\fp))^\wedge \to  (M_{i-1,{\fn_\fp}}(\fp))^\wedge.$$
Since $ M_{i,{\fn_\fp}}(\fp)$ is a finite $R_{\fn_\fp}$-module  we have
 (by \cite[Theorem 10.13]{AM69})
  $ \hat M_{i} (\fp) = M_{i,{\fn_\fp}}(\fp) \otimes_{R_{\fn_\fp}} R_{\fn_\fp}^\wedge$.  Moreover $R_{\fn_\fp}^\wedge$ is faithfully flat over the local ring $R_{\fn_\fp}$ (\cite[Theorem 8.14]{Mat86}). Hence 
  $F_{i\fn_\fp}: M_{i-1,{\fn_\fp}}(\fp) \to M_{i-1,{\fn_\fp}}(\fp)$ 
  is injective iff $F_i(\fp): \hat M_{i-1} (\fp) \to \hat M_{i-1}(\fp)$ is injective.
 
(ii) By assumption $K_{i-1}=0$ for $i=1,\ldots, m$ and $M/\langle F_{1},\ldots,F_{m}\rangle M$ is $A$-flat. 
By Lemma \ref{lem.maxA} the Jacobson radical of $R$  contains $\langle x \rangle$ and $F_{1},\ldots,F_{m}$ is a regular sequence contained in the Jacobson radical. Hence $M_{i}$ is $A$-flat and $M_{i,{\fn_\fp}}$ is $A_\fp$-flat for all $i$ (repeated application of \cite[Theorem 22.2]{Mat86}).
Tensoring  $0 \to M_{i-1,{\fn_\fp}} \to M_{i-1,{\fn_\fp}} \to M_{i,{\fn_\fp}} \to 0$ with $\otimes_{A_\fp} k(\fp)$ we get an exact sequence
$0 \to M_{i-1,{\fn_\fp}} (\fp) \to M_{i-1,{\fn_\fp}}(\fp)) \to M_{i,{\fn_\fp}}(\fp) \to 0$ 
for  $i=1,\ldots m$ by  \cite[Theorem 22.3]{Mat86}.  Now apply (i).

(iii) Localizing the exact sequence (\ref{*}) at $\fn_\fp$ we get an exact sequence of finite $R_{\fn_\fp} $-modules. Taking the $\x$-adic completion, 
the sequence stays exact and we see that $(K_{i-1,{\fn_\fp}})^\wedge = \ker \big( F_{i\fn_\fp}^\wedge :  (M_{i-1,{\fn_\fp}})^\wedge \to (M_{i-1,{\fn_\fp}})^\wedge\big).$ 
By Lemma \ref{lem.loc} $\hat M_{i-1}(\fp) = (M_{i-1,{\fn_\fp}})^\wedge \otimes_{A_\fp} k(\fp)$ and $F_i(\fp) =  F_{i\fn_\fp}^\wedge \otimes_{A_\fp} k(\fp)$, and by assumption $F_i(\fp)$ is injective.
We  apply now repeatedly \cite[Theorem 22.5 ]{Mat86} to 
$A_\fp \to \hat R_{\fn_\fp} = A_\fp[[x]] $ and to $F_{i\fn_\fp}^\wedge$
to get that  $(K_{i-1,{\fn_\fp}})^\wedge =
K_{i-1,{\fn_\fp}} \otimes_{R_{\fn_\fp}} R_{\fn_\fp}^\wedge =0$ 
and that $(M_{i,{\fn_\fp}})^\wedge = M_{i,{\fn_\fp}} \otimes_{R_{\fn_\fp}} R_{\fn_\fp}^\wedge$ is flat over $A_\fp$ for all $i$. 
Since $R_{\fn_\fp}^\wedge$ is faithfully flat over $R_{\fn_\fp}$
this implies $K_{i-1,{\fn_\fp}} =0$ and that 
$M_{i,{\fn_\fp}}$ is flat over $A_\fp$. 

The support of the $R$-module $K_{i-1}$ is closed and hence
 $(K_{i-1})_{\fn_\fq}^\wedge=0$  for $\fq$ in an open neighbourhood $U$ of $\fp$ in $\Spec A$. Moreover the flatness of  $M/\langle F_{1},\ldots,F_{m}\rangle M$ implies that $M_{\fn_\fq}^\wedge/\langle F_{1},\ldots,F_{m}\rangle M_{\fn_\fq}^\wedge$ is $A_\fq$-flat. 
Applying \cite[Theorem 22.5 ]{Mat86} now to 
$F_{i\fn_\fq}^\wedge :  (M_{i-1})_{\fn_\fq}^\wedge \to (M_{i-1})_{\fn_\fq}^\wedge $
we get that  $\hat M_{i-1}(\fq) \to  \hat M_{i-1}(\fq)$ is injective and that 
$F_1(\fq),\ldots,F_m(\fq)$ is an $\hat M(\fq)$-sequence. 
\end {proof}

\begin{proposition} \label{prop.icis}
Let $A$ be a Noetherian ring and  $I \subset \x A[[x]]$ an ideal generated by 
 $F_1,\ldots, F_m $, such that $ A[[x]]/I A[[x]]$ is $A$-flat.
  For $\fp \in \Spec A$ denote  by $\hat I(\fp) \subset  k(\fp)[[x]]$  the ideal generated by $F_1(\fp),\ldots,F_m(\fp)$.

\begin{enumerate}
\item  If  $\hat I(\fp)$ is a complete intersection, then $\hat I(\fq)$ is a complete intersection for $\fq$ in an open neighbourhood of $\fp$ in $\Spec A$. 

\item  Assume that $\hat I(\fp)$ is an ICIS and that the hypotheses of Proposition \ref{prop.TI} are satisfied. 
Then  $\hat I(\fq)$ is an ICIS with $\tau(I(\fp)) \geq \tau(I(\fq))$ for $\fq$ in an open neighbourhood of $\fp$ in $\Spec A$.
\end{enumerate}
\end{proposition}

\begin{proof} 1. We may assume that $F_1(\fp),\ldots,F_m(\fp)$ is a $k(\fp)[[x]]$-sequence. By Proposition \ref{prop.reg}  $F_1(\fq),\ldots,F_m(\fq)$ is a $k(\fq)[[x]]$-sequence, hence $\hat I(\fq)$ is a complete intersection, for $\fq$ in an open neighbourhood of $\fp$ in $\Spec A$.\\
2. follows from Proposition \ref{prop.TI} since  for $\hat I(\fq) \subset k(\fq)[[x]]$ a complete intersection  $\dim_{k(\fq)} T_{\hat I(\fq)}$ is the Tjurina number of $\hat I(\fq)$.
\end {proof}


\end{document}